\documentclass[10pt]{amsart}

\usepackage{graphicx}
\usepackage{amsfonts,amsmath,latexsym,amssymb,amsthm}
\usepackage{verbatim}
\usepackage{mathpazo}

\newcounter{theoremcounter}
\newcounter{lemmacounter}

\newcounter{dummycounter}

\newcounter{corcounter}

\newcounter{emptycounter}
\newcounter{defcounter}

\newtheorem{theorem}[theoremcounter]{Theorem}

\newtheorem{lemma}[lemmacounter]{Lemma}

\newtheorem{proposition}[lemmacounter]{Proposition}
\newtheorem{corollary}[corcounter]{Corollary}

\newtheorem{definition}[defcounter]{Definition}

\numberwithin{equation}{section}
\numberwithin{lemmacounter}{section}
\numberwithin{propcounter}{section}
\numberwithin{corcounter}{section}
\numberwithin{conjcounter}{section}
\numberwithin{theoremcounter}{section}
\numberwithin{probcounter}{section}

\newcounter{eqncounter}

\numberwithin{equation}{eqncounter}

\def\IR{\mathbb R}
\def\IC{\mathbb C}
\def\IZ{\mathbb Z}

\def\IQ{\mathbb Q}

\def\P{\mathcal{P}}

\def\kbar{\overline{k}}
\def\coll{\mathcal{C}}
\def\M{M}
\def\M{M}
\def\S_I{Z_I}
\def\ST{Z(T)}
\def\SXd{Z(X^d)}
\def\SXm{Z(X^m)}
\def\Sempty{Z_{\emptyset}}

\def\valpha{{\mbox{\boldmath $\alpha$}}}

\def\z{{\bf z}}
\def\x{{\bf x}}
\def\y{{\bf y}}

\def\Oseen{{\mathcal{O}}}
\def\A{{\mathfrak{A}}}
\def\C{{\mathfrak{C}}}
\def\c{{\mathfrak{c}}}

\def\E{E}
\def\F{{F}}

\def\m{{\mathfrak{m}}}
\let\rho\varrho

\def\ti{{\tau_{{\bf j}}}}
\def\Ti{{\phi_{{\bf j}}}}
\def\i{{{\bf j}}}
\def\0{{{\bf 0}}}
\def\tr{tr_{{\bf j}}}

\def\vx{{\bf x}}
\def\vy{{\bf y}}

\def\mug{\mu_g}
\def\mun{\eta}
\def\eO{\eta_{\Oseen}}

\def\vdelta{{\mbox{\boldmath $\delta$}}}

\def\Vol{\textup{Vol}}
\def\Pic{\textup{Pic}}

\def\Lam1{\Lambda_1}

\def\q{q'}
\def\r{r'}
\def\s{s'}

\def\Ci{C_{\bf{j}}}
\def\C0{C_{\bf{0}}}
\def\Di{{F}_{\bf{j}}}
\def\Da{D}
\def\SIi0{S_{F}(\j)}
\def\SiF{Z_{\Di}}
\def\SibF{Z_{{\F}_{\j}}}
\def\SI0i{S_{F}({\bf{0}})}

\def\psiTi{\psi_{\bf j}}
\def\j{{\bf{j}}}

\def\Log{{\log^+}}
\def\c1{c_1}

\def\ka{\kappa}
\def\Cem{C_{e,m}}

\def\N{\mathfrak{N}}
\def\Ze{\mathcal{Z}_1}
\def\Zz{\mathcal{Z}_2}
\def\Zi{\mathcal{Z}_p}

\def\BK{B_K}
\def\Bk{B_k}

\begin{document}
\title{Integral points of fixed degree and bounded height}

\author{Martin Widmer}

\address{Department of Mathematics\\ 
Royal Holloway, University of London\\ 
TW20 0EX Egham\\ 
UK}

\email{martin.widmer@rhul.ac.uk}

\date{\today}



\maketitle

\begin{abstract}
By Northcott's Theorem there are only finitely many algebraic points in affine $n$-space of fixed degree $e$ over a given number field and of height at most $X$. Finding the asymptotics for these cardinalities as $X$ becomes large is a long standing problem
which is solved only for  $e=1$ by Schanuel, for $n=1$ by Masser and Vaaler, 
and for $n$ ``large enough''  by Schmidt, Gao, and the author. In this paper
we study the case where the coordinates of the points are restricted to algebraic integers,
and we derive the analogues of Schanuel's, Schmidt's, Gao's and the author's results.
The proof invokes tools from dynamics on homogeneous spaces, algebraic number theory,
geometry of numbers, and a geometric partition method due to Schmidt. 
\end{abstract}

\section{Introduction}\label{intro}

In this article we count algebraic points of bounded Weil height with integral coordinates, generating an extension of given degree over a fixed number field. 

Let $k$ be a number field, let $\kbar$ be an algebraic closure of $k$, and
let $H$ be the absolute multiplicative (affine) Weil height on $\kbar^n$ (for the definition see (\ref{definitionheight}) below).
One of the most fundamental and important properties of the height 
asserts that subsets of $\kbar^n$ of uniformly bounded height and degree are finite.
This result was shown by Northcott  \cite{Northcott50} in 1950, and his proof provides explicit upper bounds. 
However, for big $n$ these
estimates are rather poor, and even nowadays,  the correct order of magnitude is known only in some special cases.

In 1962 Lang \cite{LangDGold} proposed the problem of asymptotically counting points of bounded height
in a fixed number field, i.e., to count points in $\kbar^n$ of degree $1$ over $k$. 
This problem has been solved by Schanuel \cite{Schanuel1} in 1964, with a detailed proof \cite{25} published 15 years later. 
The problem of counting points of bounded height and of
fixed degree $e>1$ over a given number field $k$ appears to be much more difficult.
Indeed, it took over $40$ years before the first significant improvement of Northcott's Theorem was established.
In 1991 Schmidt \cite{22} obtained upper bounds that greatly improved upon Northcott's bounds. 
However, when the degree and the dimension are both bigger than $1$ Schmidt's bounds are still significantly larger than what one expects.
Later, in \cite{14} Schmidt established the asymptotics
for points quadratic over $\IQ$, and this in all dimensions $n$.
This in turn yield new results on a generalized version of Manin's conjecture (the special case $n=2$ provides one of the rare examples of a cubic four fold for which the Batyrev-Manin conjecture is established and, as observed by Le Rudulier \cite{LeRudulier},  leads to a counterexample to Peyre's predicted constant).  
Soon afterwards Gao \cite{7} gave asymptotics for points in $n$ dimensions of degree $e$ over $\IQ$, subject to the 
constraint $n>e$.
The case $n=1$ was treated by Masser and Vaaler in \cite{37}, and was generalized in \cite{1} by the same authors to allow arbitrary ground fields $k$.
The author \cite{art2} has established asymptotic estimates for points in $n$ dimensions of fixed degree $e$ over an arbitrary number field, provided $n>5e/2+5$. A short survey on counting points of fixed degree is given in Section 4 of Bombieri's
article \cite{BomSur}.

Regarding integral points of fixed degree $e>1$ the subject is less developed.
For a number field $k$ let us write $N(\Oseen_{k}(n;e),X)$ for the number of points 
$\valpha=(\alpha_1,\ldots,\alpha_n)$ of absolute multiplicative Weil height 
no larger than $X$, whose coordinates are algebraic
integers with $[k(\alpha_1,\ldots,\alpha_n):k]=e$.
In \cite[p.81]{3} Lang has stated without proof  
\begin{alignat}1\label{Lang}
N(\Oseen_k(1;1),X)=\gamma_kX^m(\log X)^{q_k}+O(X^m(\log X)^{q_k-1}).
\end{alignat}
Here $m=[k:\IQ]$, $q_k$ is the rank of the group of units and $\gamma_k$ is an unspecified positive constant 
depending on $k$. The formula (\ref{Lang}) can easily be deduced from a counting principle of Davenport \cite{15},
but it is not  a straightforward application of counting lattice points in homogeneously expanding domains (cf. \cite[p.81]{3}).
The asymptotics for $N(\Oseen_k(n;1),X)$ can also be obtained from \cite[Theorem 3.11.3]{ChambertLoirTschinkel12}. 
Regarding higher degrees Chern and Vaaler \cite{43} proved asymptotic estimates for the number of monic polynomials of fixed degree with rational integral coefficients and 
bounded Mahler measure. As these estimates are of polynomial growth, and since the Mahler measure is multiplicative, one can easily see that the reducible polynomials 
do not effect the asymptotics.
Thus Chern and Vaaler's result implies asymptotics for $N(\Oseen_{\IQ}(1;e),X)$. More precisely, their Theorem 6 yields
\begin{alignat}1\label{ChernVaaler}
N(\Oseen_{\IQ}(1;e),X)=c_eX^{e^2}+O(X^{e^2-1}),
\end{alignat}
with a positive and explicit constant $c_e$ depending on $e$. Very recently, Barroero \cite{Barroero13} has generalized 
(\ref{ChernVaaler}) to arbitrary ground fields $k$, and then further generalized  this to $S$-integers \cite{Barroero14}.
Barroero's approach follows the one in \cite{1} of counting polynomials of degree $e$. This strategy
is more straightforward and easier than ours but, unfortunately, works only for $n=1$. 

One of our goals here is to deduce statements about points with integral coordinates
analogous to the results of Schanuel, Schmidt, Gao, and the author alluded
to above. This is the first attempt to prove asymptotic estimates for 
$N(\Oseen_{k}(n;e),X)$ with the exception of the special cases $e=1$ or $n=1$.\\

Another new aspect of this article is that our methods allow us to prove a multi-term expansion of $N(\Oseen_{k}(n;e),X)$.
For instance, we are able to find the first $q_k+1$ leading terms in (\ref{Lang}), and an error term of order $X^{m-1}(\log X)^{q_k}$.
This is in contrast to the results on points of fixed degree, mentioned in the previous paragraph.
The $q_k+1$ different main terms of decreasing order have a simple geometric interpretation which we shall
explain later in Section \ref{generalisation}. The main terms can be expressed using Laguerre polynomials, e.g.,  
\begin{alignat}1\label{app1}
N(\Oseen_k(n;1),X)=\Bk^n X^{mn}L_{q_k}(-\log X^{mn})+O(X^{mn-1}(\log X)^{q_k}). 
\end{alignat}
Here $L_{q_k}(x)$ is the $q_k$-th Laguerre polynomial, and $\Bk$ is a field invariant defined later on. 
The somewhat unexpected appearance of the Laguerre polynomial in the main term is another new feature of our result.

It is typical with these types of asymptotic expansions for the main term to be of the form
$X^{a}P(\log X)$ for some polynomial $P(x)$. This polynomial is often obtained via a meromorphic
continuation of the corresponding height zeta function and a suitable Tauberian theorem; see, e.g., Franke, Manin, and Tschinkel's pioneering article \cite[Corollary]{38} for the case of rational points on Flag manifolds $V$. In their case the degree $\deg P$ is also related to the rank of a group, more precisely,
$\deg P$ is the rank of the Picard group $\Pic(V)$ minus $1$\footnote{There is a misprint in their Corollary, $t$ should read $t-1$.}.
Franke, Manin, and Tschinkel obtained their result by expressing the corresponding height zeta function as an Eisenstein series and then using Langland's work to 
study its analytic properties. Similar, technically intricate, methods have been used in \cite{ChambertLoirTschinkel10} and \cite{ChambertLoirTschinkel12}. 
Our proof makes no use of complex analysis. Indeed, we reverse the situation here, and we say
something about the analytic properties of the height zeta function $\zeta_{k,n,e}(s)=\sum_{\valpha\in \Oseen_{k}(n;e)}H(\valpha)^{-s}$ 
using our estimates for $N(\Oseen_k(n;e),X)$.\\

To state our first result we need some notation. 
Let $K\subset \kbar$ be a number field, write $d=[K:\IQ]$ for its degree, and
let $M_K$ denote the set of places of $K$. For each place $v$ we choose the unique representative $|\cdot|_v$ that either extends the usual Archimedean absolute value on $\IQ$ or a usual $p$-adic absolute value on $\IQ$. Let $K_v$ be the completion of $K$ with respect to $v$, and let
$\IQ_v$ the completion with respect to the place of $\IQ$ below $v$, and write $d_v=[K_v:\IQ_v]$ for the local degree at $v$. For a point $\valpha\in K^n$ we define the absolute multiplicative (affine) Weil height of $\valpha$ as
\begin{alignat}1\label{definitionheight}
H(\valpha)=\prod_{v\in M_K}\max\{1,|\alpha_1|_v,\ldots,|\alpha_n|_v\}^{\frac{d_v}{d}}.
\end{alignat}
As is well-known $H(\valpha)$ is independent of the number field $K$ containing the coordinates $\alpha_i$, and hence 
$H(\cdot)$ defines a genuine function on $\kbar^n$.

For a subset $S$ of $\kbar^n$
of uniformly bounded degree and real numbers $X\geq 1$ we define the counting function 
\begin{alignat*}1
N(S,X)=|\{\valpha\in S;H(\valpha)\leq X\}|.
\end{alignat*}
Thanks to Northcott's Theorem the quantity above is finite for each $X$.
For positive rational integers $e$ and $n$ we define the set of integral points in $n$ dimensions of degree $e$ over the field $k$
\begin{alignat*}1
\Oseen_k(n;e)=\{\valpha\in \Oseen_{\kbar}^n; [k(\valpha):k]=e\}.
\end{alignat*}
Here $\Oseen_{\kbar}\subset \kbar$ denotes the ring of algebraic integers, and $k(\valpha)=k(\alpha_1,\ldots,\alpha_n)$.
Let $\coll_e(k)$ be the collection of all field extensions of $k$ of degree $e$, i.e.,
\begin{alignat*}1
\coll_e(k)=\{K\subset \kbar; [K:k]=e\}.
\end{alignat*}
For a number field $K$ we write $\Delta_K$ for the discriminant of $K$,
$r_K$ for the number of real, $s_K$ for the number of pairs of complex conjugate embeddings of $K$, and $q_K=r_K+s_K-1$ for the rank of the
group of units. Moreover, we set
\begin{alignat*}1
t_e(k)&=\sup\{q_K;K\in \coll_e(k)\}=e(q_k+1)-1,\\
\BK&=\frac{2^{r_K}(2\pi)^{s_K}}{\sqrt{|\Delta_K|}},
\end{alignat*}
and for $0\leq i\leq t_e(k)$ we introduce the formal sum
\begin{alignat}1\label{Dconstant}
D_i=D_i(k,n,e)=\sum_{K\in \coll_e(k) \atop q_K\geq i}\frac{\BK^n}{i!}{q_K\choose i}.
\end{alignat}
For $e>1$ we define
\begin{alignat*}1
\Cem=\max\{2+\frac{4}{e-1}+\frac{1}{m(e-1)},7-\frac{e}{2}+\frac{2}{me}\}\leq 7.
\end{alignat*}
Finally, we put $\Log X=\max\{1,\log X\}$.
Now we can state our first result.
\begin{theorem}\label{mainthm1}
Let $k$ be a number field and $m=[k:\IQ]$.
Suppose that either $e=1$ or that $n>e+\Cem$, and set $t=t_e(k)$. Then the sums in (\ref{Dconstant}) converge, and for $X\geq 1$ we have 
\begin{alignat}1\label{mainthm1ineq}
\left|N(\Oseen_k(n;e),X)-\sum_{i=0}^{t}D_iX^{men}(\log X^{men})^i\right|\leq c_1X^{men-1}(\log^+ X)^{t}
\end{alignat}
for some positive constant $c_1=c_1(n,m,e)$ depending only on $n$, $m$ and $e$.
\end{theorem}
We remark that the sum in (\ref{mainthm1ineq}) can be written as the weighted sum of Laguerre polynomials 
$X^{men}\sum_q \beta_q L_{q}(-\log X^{men})$.
Here $q$ runs over the finite set $\{q_K;K\in \coll_e(k)\}$, and $\beta_q=\beta_q(k,e,n)=\sum_K \BK^n$, where the sum is taken over all
$K\in \coll_k(e)$ with $q_K=q$.

Note that for $e\geq 9$ the condition $n>e+\Cem$ is equivalent to $n>e+2$.
Unfortunately, this is probably not the sharp bound.  
However, as 
$N(\Oseen_k(1;e),X)\leq N(\Oseen_k(n;e),X)$ we see by comparing with (\ref{ChernVaaler}) that if $m=1$ then (\ref{mainthm1ineq}) 
cannot hold for $n<e$.
Borrowing ideas of Masser and Vaaler from \cite{1}, Theorem \ref{mainthm1}, combined with standard estimates for the Mahler measure, shows that
$N(\Oseen_{k}(1;e),X)\gg X^{me^2}(\log X)^{q_k}$. Hence, (\ref{mainthm1ineq}) cannot hold for $n<e$, even if $m>1$.
Note also that for $e=n=2$ the sums in (\ref{Dconstant}) diverge.

Next  let us choose $e=1$. Then we get the formula (\ref{app1}) which is a new result, even for $n=1$. 
Here the multi-term expansion could probably be worked out from the results in \cite{ChambertLoirTschinkel10}, but it is unlikely that the same
error term can be obtained.  

It is probably not too difficult to extend our theorem to the context of Lipschitz heights as in
\cite{1} or even adelic Lipschitz heights as in \cite{art2}. These generalizations would have further applications
such as refined asymptotic estimates for $N(\Oseen_{k}(1;e),X)$, improving upon Barroero's result, or for the number of integral
solutions of fixed degree to a system of linear equations, analogous to the main result in \cite{art3}.
However, to keep the technical difficulties and the required notation at a minimal level, and to emphasize the main ideas and novelties of this work, 
we decided not to include
these generalizations.\\

Let us formally define the height zeta function of $\Oseen_{k}(n;e)$ as 
\begin{alignat*}1
\zeta_{k,n,e}(s)=\sum_{\valpha\in \Oseen_{k}(n;e)}H(\valpha)^{-s}.
\end{alignat*}
The upper bound of order $X^{men}(\log X)^t$ implies that $\zeta_{k,n,e}(s)$ converges in the complex half plane $\Re(s)>men$. 
But Theorem \ref{mainthm1} implies also that $\zeta_{k,n,e}(s)$ has a meromorphic continuation to $\Re(s)>men-1$ with a 
pole at $s=men$ of order $t+1$. More precisely,
setting $D_{t+1}=0$, and using summation by parts, we find that the principal part of the Laurent series at $s=men$ is given by
\begin{alignat*}1
\sum_{i=1}^{t+1}\frac{(men)^i(i-1)!(D_{i-1}+iD_i)}{(s-men)^{i}}.
\end{alignat*}
Theorem \ref{mainthm1} will be proved via our main result Theorem \ref{mainthm2} which we present in the next section.
 
 \section{The main result}\label{generalisation}

Suppose $K$ is a field extension of $k$ of degree $e=[K:k]$, and put $[K:\IQ]=d$, so that $d=em$. 
We denote by $\sigma_1,\ldots,\sigma_d$ 
the embeddings from $K$ to $\IR$ or
$\IC$ respectively, ordered such that
$\sigma_{r+s+i}=\overline{\sigma}_{r+i}$ for 
$1\leq i \leq s$, i.e., $\sigma_{r+s+i}$ and $\sigma_{r+i}$ are complex conjugate.
Let $\Oseen$ be a submodule of the free $\IZ$-module $\Oseen_K$ of full rank.
Let $\A_{\Oseen}$ be the smallest ideal in $\Oseen_K$ that contains $\Oseen$, i.e., $\A_{\Oseen}$ is the intersection
over all ideals in $\Oseen_K$ that contain $\Oseen$.
Set
\begin{alignat}1\label{defeO}
\eO=\N(\A_{\Oseen})^{1/d} \geq 1,
\end{alignat}
where $\N(\A)=|\Oseen_K/\A|$ denotes the norm of a nonzero ideal $\A$ of $\Oseen_K$. 
Furthermore, we define
\begin{alignat*}1
G(K/k)=\{[K_0:k];k\subset K_0\subsetneq K\}
\end{alignat*}
if $K\neq k$, and we put
\begin{alignat*}1
G(K/k)=\{1\}
\end{alignat*}
if $K=k$. 
Then for an integer $g\in G(K/k)$ we define
\begin{alignat*}3
\delta_g(K/k)=\inf\{H(\alpha,\beta); k(\alpha,\beta)=K, [k(\alpha):k]=g\},
\end{alignat*}
and we set
\begin{alignat}1\label{defmug}
\mug=mn(e-g)-1.
\end{alignat}
We remark that $\delta_g(K/k)$ refines the invariant $\delta(K)$ introduced by Roy and Thunder \cite{8}. 
For a point $\valpha \in \kbar^n\backslash \{\0\}$ we write $k(\ldots,\alpha_i/\alpha_j,\ldots)$ for the extension
of $k$ generated by all possible ratios $\alpha_i/\alpha_j$ $(1\leq i,j\leq n, \alpha_j\neq 0)$ of the coordinates of $\valpha$.
Next we introduce the set of ``projectively primitive'' points in $\Oseen^n$
\begin{alignat*}1
\Oseen^n(K/k)=\{\valpha\in \Oseen^n\backslash\{\0\};K=k(\ldots,\alpha_i/\alpha_j,\ldots)\}.
\end{alignat*}
Note that for $n=1$ the set $\Oseen^n(K/k)$ is empty
if $K\neq k$ and equals $\Oseen\backslash\{0\}$ if $K=k$.
For a subset $I\subset\{1,\ldots,r_K+s_K\}$ and $I^c=\{1,\ldots,r_K+s_K\}\backslash I$ we define
\begin{alignat*}1
\Oseen_I^n(K/k)=\{\valpha\in \Oseen^n(K/k); &|\sigma_i(\valpha)|_{\infty}\geq 1 \text{ for }i\in I,\text{ and }\\
&|\sigma_i(\valpha)|_{\infty}<1 \text{ for }i\in I^c\},
\end{alignat*}
where $|\sigma_i(\valpha)|_{\infty}=\max\{|\sigma_i(\alpha_1)|,\ldots, |\sigma_i(\alpha_n)|\}$.
Finally, let $\S_I(T)$ be the measurable set in Euclidean space, defined in (\ref{defSI}), and set $\q=|I|-1$. In 
Section \ref{SectionVolComp} we will show that for $X\geq 1$
\begin{alignat*}1
\Vol \S_I(X^d)=(2^{r_K}\pi^{s_K})^n(-1)^{\q}\left(-1+X^{dn}\sum_{i=0}^{\q}\frac{(-\log X^{dn})^i}{i!}\right).
\end{alignat*}
Recall that $K/k$ is an extension of number fields and $d=[K:\IQ]$. 
We can now state the main result of this article. All our results will
be deduced from this theorem.
\begin{theorem}\label{mainthm2}
Suppose $\q=|I|-1\geq 0$, $X\geq 1$ and either $n>1$ or $K=k$. Then
\begin{alignat*}1
\left|N(\Oseen_I^n(K/k),X)-\frac{2^{s_Kn}\Vol \S_I(X^d)}{(\sqrt{|\Delta_K|}[\Oseen_K:\Oseen])^n}\right|
\leq c_2\sum_{g\in G(K/k)}\frac{X^{dn-1}(\Log X)^{\q}}{\eO^{dn-1}\delta_g(K/k)^{\mug}},
\end{alignat*}
where $c_2=c_2(n,d)$ is a positive constant depending only on $n$ and $d$. 
\end{theorem}
Using $P_{\q}(x)=\sum_{i=0}^{\q}\frac{x^i}{i!}$ we can rewrite the main term as
\begin{alignat*}1
\left(\frac{\BK}{[\Oseen_K:\Oseen]}\right)^n(-1)^{\q}\left(X^{dn}P_{\q}(-\log X^{dn})-1\right).
\end{alignat*}
Note that this expression depends only on the cardinality of $I$ but not on the particular choice of $I$ itself.
Next let us consider some special cases.
We start with the case $K=k$, i.e., $d=m$. Then the statement takes the form
\begin{alignat*}1
\left|N(\Oseen_I^n(k/k),X)-\frac{2^{s_kn}\Vol \S_I(X^m)}{(\sqrt{|\Delta_k|}[\Oseen_k:\Oseen])^n}\right|
\leq c_2(n,m)\frac{X^{mn-1}(\Log X)^{\q}}{\eO^{mn-1}}.
\end{alignat*}
Now we take $n=1$, $\Oseen=\Oseen_k$, and let us assume $r_k\geq 1$. If we choose $I=\{1\}$ and assume $m>1$,
then $N(\Oseen_I^n(k/k),X)=N(\Oseen_I,X)$ counts the primitive Pisot numbers in the real field $\sigma_1(k)$. Here 
the primitivity 
is induced by the choice of the set $I$.
The non-primitive
Pisot numbers lie in a strict subfield of $\sigma_1(k)$, and so their number has order of magnitude at most $X^{m/2}$. 
Thus for the total number of Pisot numbers in $\sigma_1(k)$ of height no larger than $X$ we get
\begin{alignat*}1
\Bk X^m+O(X^{m-1}).
\end{alignat*}
Still with $K=k$, $\Oseen=\Oseen_k$, and $n=1$ we now take $I=\{1,\ldots,r_k+s_k\}$. Then we are counting 
the nonzero elements $\alpha\in \Oseen_k$ with $H(\alpha)=|Nm_{k/\IQ}(\alpha)|^{1/m}\leq X$. Their number is given by
\begin{alignat*}1
\sum_{i=0}^{q_k}(-1)^{q_k}\Bk X^m\frac{(-\log X^m)^i}{i!}
+O(X^{m-1}(\Log X)^{q_k}).
\end{alignat*}
Next note that 
\begin{alignat}1\label{Oseenunion}
\Oseen^n(K/k)=\cup_I \Oseen_I^n(K/k), 
\end{alignat}
taken over all non-empty subsets of $I$ of $\{1,\ldots,r_K+s_K\}$, is
a disjoint union. 
Thus we may sum the estimate in Theorem \ref{mainthm2} over all non-empty sets $I$ to get estimates for the counting function of
$\Oseen^n(K/k)$. We even get a geometric interpretation of the main terms. The highest order main term comes from the points in $\Oseen_I^n(K/k)$ with maximal 
$I$, i.e., points satisfying $|\sigma_i(\valpha)|_{\infty}\geq 1$ for all $i$.
For the second order main term there is a negative contribution from $\Oseen_I^n(K/k)$ with maximal $I$ and a positive contribution for
each $\Oseen_I^n(K/k)$ with $|I|=r_K+s_K-1$, and so forth. \\
 
Let $\ST=\cup_I \S_I(T)$, where this time $I$ runs over all subsets of $\{1,\ldots,r_K+s_K\}$, and again this is a disjoint union.
In Section \ref{SectionVolComp} we will show that for $X\geq 1$
\begin{alignat*}1
\Vol \SXd=(2^{r_K}\pi^{s_K})^nX^{dn}\sum_{i=0}^{q_K}\frac{(\log X^{dn})^i}{i!}{q_K\choose i}=(2^{r_K}\pi^{s_K})^nX^{dn}L_{q_K}(-\log X^{dn}).
\end{alignat*}
As $\Oseen_{\emptyset}^n(K/k)=\emptyset$ we see that the union in (\ref{Oseenunion}) taken over all subsets remains equals $\Oseen^n(K/k)$.
In Section \ref{Corollary2} we show that Theorem \ref{mainthm2} remains valid for $I=\emptyset$, provided $(\Log X)^{\q}$
in the error term is replaced by $1$.
From this and Theorem \ref{mainthm2} we may deduce the following result.
\begin{corollary}\label{cor2}
Suppose $X\geq 1$ and either $n>1$ or $K=k$. Then
\begin{alignat*}1
\left|N(\Oseen^n(K/k),X)-\frac{2^{s_Kn}\Vol \SXd}{(\sqrt{|\Delta_K|}[\Oseen_K:\Oseen])^n}\right|
\leq c_3\sum_{g\in G(K/k)}\frac{X^{dn-1}(\Log X)^{q_K}}{\eO^{dn-1}\delta_g(K/k)^{\mug}},
\end{alignat*}
where $c_3=c_3(n,d)$ is a positive constant depending only on $n$ and $d$.
\end{corollary}
Note that here, opposed to in Theorem \ref{mainthm2}, all main terms are positive.
Let us briefly explain the strategy of the proof of Theorem \ref{mainthm1}.
To this end we define the set of
``non-projectively primitive'' points in $\Oseen_K^n$
\begin{alignat*}1
\Oseen_{npp}^n(K/k)=\{\valpha\in \Oseen_K^n\backslash \Oseen_K^n(K/k); k(\valpha)=K\}.
\end{alignat*}
Now any $\valpha$ in $\Oseen_k(e,n)$ lies either in $\Oseen_K^n(K/k)$ or in $\Oseen_{npp}^n(K/k)$, with $K=k(\valpha)\in \coll_e(k)$.
Hence we have the following disjoint union
\begin{alignat*}1
\Oseen_k(e,n)=\bigcup_{\coll_e(k)}\Oseen_K^n(K/k)\cup \Oseen_{npp}^n(K/k).
\end{alignat*}
Therefore, we just have to sum $N(\Oseen_K^n(K/k),X)$ and $N(\Oseen_{npp}^n(K/k),X)$ over all $K$ in $\coll_e(k)$.
And indeed, we will show that the sum over all main terms as well as the sum over all error terms of $N(\Oseen_K^n(K/k),X)$ converges,
provided $n>e+\Cem$, while the sum over $N(\Oseen_{npp}^n(K/k),X)$ has smaller order of magnitude. 

It now is obvious that a crucially important feature of Corollary \ref{cor2} (and so of Theorem \ref{mainthm2}) 
is the good dependence of the error term on the extension $K/k$; note that by Northcott's Theorem
$\delta_g(K/k)^{-\mug}$ tends to zero as $K$ runs over the subset $\coll_e^{(g)}(k)$ of those $K\in \coll_e(k)$ with $g\in G(K/k)$. To compare with 
the discriminant we can apply 
a well-known inequality of Silverman \cite[Theorem 2]{9} to get $\delta_g(K/k)\geq c_k |\Delta_K|^{1/(2me(e-1))}$ for some positive constant $c_k$.

Unfortunately, bounding the number of extensions $K/k$ of fixed degree $e$ and bounded discriminant is a 
difficult problem, satisfactorily solved only for $e\leq 5$, thanks to the deep work of Datskowsky and Wright \cite{82}, and Bhargava \cite{Bhargava05, Bhargava10}.
We surmount this impasse by deviating from the standard route and working with the new invariant
$\delta_g(K/k)$ instead of the classical discriminant.
As it turns out we have almost sharp bounds for
the number of fields  $K\in\coll_e^{(g)}(k)$ with  $\delta_g(K/k)\leq T$, opposed to the case when we enumerate by the discriminant. Furthermore,
as larger $g$ gets, which means as larger the error terms get, the better our upper bounds for the number of $K\in\coll_e^{(g)}(k)$ with  $\delta_g(K/k)\leq T$
become. These observations have already been used in \cite{art2}.

Our method leads also to asymptotics for more specific sets, e.g., points $\valpha$ of degree $d$ whose coordinates are primitive Pisot numbers
of $\IQ(\valpha)$, provided $n>d+\Cem+1$. Here the ``$+1$'' is required to exclude the points with some coordinates equal zero.\\

The special case $K=k$ in Corollary \ref{cor2} yields a generalization of (\ref{app1}) (to arbitrary submodules of $\Oseen_k$
of full rank) with a more precise error term.
We have
\begin{alignat}1\label{app2}
\left|N(\Oseen^n\backslash\{\0\},X)-\frac{2^{s_kn}\Vol \SXm}{(\sqrt{|\Delta_k|}[\Oseen_k:\Oseen])^n}\right|
\leq c_3(n,m)\left(\frac{X}{\eO}\right)^{mn-1}(\Log X)^{q_k}.
\end{alignat}
Now let us choose $\Oseen=\A$ for a nonzero ideal $\A$. Then we have $\eO=\N(\A)^{1/m}$.
This allows one to carry out a M\"obius inversion to count $\valpha \in \A^n$ satisfying another type of primitivity, namely
$\alpha_1\Oseen_k+\cdots+\alpha_n\Oseen_k=\A$. Here we need $n\geq 2$ to get for the number of such $\valpha$
\begin{alignat*}1
\frac{2^{s_kn}\Vol \SXm}{\zeta_k(n)(\sqrt{|\Delta_k|}\N(\A))^n}
+O\left(\frac{X^{mn-1}(\Log X)^{\overline{q}}}{\N(\A)^{n-1/m}}\right),
\end{alignat*}
where $\overline{q}=q_k$ if $(n,m)\neq (2,1)$ and $\overline{q}=1$ if $(n,m)=(2,1)$.

\section{Techniques and plan of the paper}\label{organisation}

The paper is organized as follows. We start with a section on elementary counting principles. 
Here we recall and provide some basic results on counting lattice points.
Then in Section \ref{toolsinvprin} we state a precise estimate (Theorem \ref{mainthm}) of the quantity 
$|\Lambda\cap \S_I(T)|$, 
for lattices $\Lambda$ that have a bounded orbit under the flow induced by 
a certain subgroup $\mathcal{T}$ of the diagonal endomorphisms with determinant $1$.


In Section \ref{1subsec4} we introduce some notation and state some simple properties of 
the sets $\S_I(T)$ and $\ST$ which are required for the proof of Theorem \ref{mainthm}.


Skriganov \cite{Skriganov1, Skriganov2} obtained very good estimates for the number of lattice points inside aligned boxes, provided the lattice orbit under the above mentioned flow is bounded. However, our set $\S_I(T)$ has hyperbolic spikes and is far away from box-shaped. To overcome this hurdle we adapt a geometric partition method that goes back to Schmidt \cite{14}, and combine it with tools from dynamics on homogeneous spaces. 
An extensions of Schmidt's partition method is applied in Section \ref{decomptransSFT} and Section \ref{furthertrans}.
To apply the simple counting principles we still have to check some technical conditions such as the Lipschitz parameterizability of the boundary,
and this is done in Section \ref{SectionLippara}. 
In Section \ref{proofmainthm} we are finally in position to 
apply the elementary counting principles, and we can conclude the proof of Theorem \ref{mainthm}.\\ 

The most important aspect of Theorem \ref{mainthm2} is the good error term, in particular, with respect to the extension $K/k$.
This particular feature imposes serious additional challenges.
Instead of the boundedness of the orbit of  $\Lambda$ under the flow of $\mathcal{T}$
we have to prove that the orbit of $\Lambda$ (scaled to have determinant $1$) lies in a certain subset of the space of lattices $SL_{dn}(\IR)/SL_{dn}(\IZ)$ which is defined in terms of the higher successive minima and involves a critical successive minimum $\lambda_l$. To show that this condition 
implies the desired error term we need to utilize the machinery developed in \cite{art1}. However, the latter can only be applied to the set $\Oseen_K^n(K/k)$ of projectively primitive points, and
this is exactly why we have to restrict the counting in Theorem \ref{mainthm2} to these points. Thus, to prove Theorem \ref{mainthm1}
we have to deal with the set $\Oseen_{npp}^n(K/k)$ separately.
In Section \ref{estsuccmin} we show that the orbits of the lattices coming from embeddings of $\Oseen^n$  under the flow of $\mathcal{T}$ are bounded, and satisfy the refined conditions involving the 
higher successive minima  as well. 
The entire Section \ref{estsuccmin} is heavily based on \cite[Section 9]{art1}.
In Section \ref{nonprimpts} we prove an upper bound for the number of lattice points that are
not projectively primitive. With this upper bound we are ready in Section \ref{estproprimpts} to prove a precise asymptotic estimate for 
the number of
projectively primitive lattice points for all components that arise from the partition method. 
Section \ref{Proof of Theorem2} finishes the proof 
of Theorem \ref{mainthm2}. \\

Corollary \ref{cor2} is essentially an immediate consequence of Theorem \ref{mainthm2}. However, the present statement requires an analogue of 
Theorem \ref{mainthm2} in the case $\q=-1$.
The latter is stated an proved in Section \ref{Corollary2}. 
The volumes of the sets $\S_I(T)$ and $\ST$ are computed in Section \ref{SectionVolComp}. 
In Section \ref{nonproprimpts} we prove that the sum over $N(\Oseen_{npp}^n(K/k),X)$ taken over all fields $K\in \coll_e(k)$ is covered by the error term 
in Theorem \ref{mainthm1}.
Finally, Section \ref{proofmainthm1} is devoted to the proof of Theorem \ref{mainthm1}.\\

We will use Vinogradov's notation $\ll$. The implied constants depend only on $n$, $m$, $e$ and $d$. 
Throughout this article $T$ and $X$ denote real numbers $\geq 1$.

\section{General counting principles}\label{countingprinciples}

For a vector $\vx$ in $\IR^\Da$ we write $|\vx|$ 
for the Euclidean length of $\vx$. The closed Euclidean ball centered at $\vx$ with radius $r$
will be denoted by $B_{\vx}(r)$.
Let $\Lambda$ be a lattice of rank $\Da$ in $\IR^\Da$ then we define the \it successive minima 
\rm $\lambda_1(\Lambda),...,\lambda_\Da(\Lambda)$ of $\Lambda$ as the successive minima in the sense of Minkowski with respect
to the unit ball. That is
\begin{alignat*}3
\lambda_i=\inf \{\lambda;\, B_0(\lambda)\cap \Lambda
\text{ contains $i$ linearly independent vectors}\}.
\end{alignat*}
\begin{definition}
Let $M$ and $\Da$ be positive integers, and let $L$ be a non-negative real.
We say that a set $Z$ is in Lip$(\Da,M,L)$ if 
$Z$ is a subset of $\IR^\Da$, and 
if there are $M$ maps 
$\rho_1,\ldots,\rho_M:[0,1]^{\Da-1}\longrightarrow \IR^\Da$
satisfying a Lipschitz condition
\begin{alignat*}3
|\rho_i(\vx)-\rho_i(\vy)|\leq L|\vx-\vy| \text{ for } \vx,\vy \in [0,1]^{\Da-1}, i=1,\ldots,M 
\end{alignat*}
such that $Z$ is covered by the images
of the maps $\rho_i$. For $\Da=1$ this is to be interpreted as the finiteness of the set $Z$,
and the maps $\rho_i$ are considered points in $\IR^\Da$ such that $Z\subset \{\rho_i; 1\leq i\leq M\}$.
\end{definition}
We will apply the following counting result from \cite[Theorem 5.4]{art1}.
\begin{theorem}\label{TWi1}
Let $\Lambda$ be a lattice in $\IR^\Da$
with successive minima $\lambda_1,\ldots,\lambda_\Da$.
Let $Z$ be a bounded set in $\IR^\Da$ such that
the boundary $\partial Z$ of $Z$ is in Lip$(\Da,M,L)$.
Then $Z$ is measurable, and, moreover,
\begin{alignat*}3
\left||Z\cap\Lambda|-\frac{\Vol Z}{\det \Lambda}\right|
\leq c_4(\Da)M\max_{0\leq i<\Da}\frac{L^i}{\lambda_1\cdots\lambda_i}.
\end{alignat*}
For $i=0$ the expression in the maximum is to be understood
as $1$. Furthermore, one can choose $c_4(\Da)=\Da^{3\Da^2/2}$.
\end{theorem}
If $\Lambda$ is a lattice in $\IR^\Da$ and $a$ is an integer with $1\leq a\leq \Da$ then we put
\begin{alignat}3\label{Lambda_a}
\Lambda(a)=\{\vx\in\Lambda;|\vx|\geq \lambda_a\}.
\end{alignat}

\begin{corollary}\label{TWi1cor}
Let $\Lambda$ be a lattice in $\IR^\Da$
with successive minima $\lambda_1,\ldots,\lambda_\Da$.
Let $Z$ be a bounded set in $\IR^\Da$ such that
the boundary $\partial Z$ of $Z$ is in Lip$(\Da,M,L)$, and $Z\subset B_0(\kappa L)$ with $\kappa\geq 1$.
Then $Z$ is measurable and we have
\begin{alignat*}3
\left||Z\cap\Lambda(a)|-\frac{\Vol Z}{\det \Lambda}\right| \leq c_5(\Da)M\frac{(\kappa L)^{\Da-1}}{{\lambda_1}^{a-1}{\lambda_a}^{\Da-a}}.
\end{alignat*}
One can choose $c_5(\Da)=c_4(\Da)(2\pi\Da)^\Da$.
\end{corollary}
\begin{proof}
The measurability comes directly from Theorem \ref{TWi1}.
First suppose $\kappa L\geq \lambda_a$.
By the triangle inequality we get
\begin{alignat*}3
\left||Z\cap\Lambda(a)|-\frac{\Vol Z}{\det \Lambda}\right| \leq 
\left||Z\cap\Lambda|-\frac{\Vol Z}{\det \Lambda}\right|+|B_0(\lambda_a)\cap\Lambda|.
\end{alignat*}
We apply Theorem \ref{TWi1}.
Since $\kappa\geq 1$, we have
\begin{alignat*}1
\left||Z\cap\Lambda|-\frac{\Vol Z}{\det \Lambda}\right|\leq c_4(\Da)M\max_{0\leq i<\Da}\frac{L^i}{\lambda_1\cdots\lambda_i} 
\leq c_4(\Da)M\frac{(\kappa L)^{\Da-1}}{{\lambda_1}^{a-1}{\lambda_a}^{\Da-a}}.
\end{alignat*}
To estimate $|B_0(\lambda_a)\cap\Lambda|$ we observe that $\partial B_0(\lambda_a)$ lies in Lip$(\Da,1,2\pi\Da\lambda_a)$.
Applying Theorem \ref{TWi1} gives
\begin{alignat*}3
|B_0(\lambda_a)\cap\Lambda|\leq 
\frac{\Vol B_0(\lambda_a)}{\det \Lambda}+c_4(\Da)\max_{0\leq i<\Da}\frac{(2\pi\Da\lambda_a)^i}{\lambda_1\cdots\lambda_i}. 
\end{alignat*}
Using Minkowski's second Theorem we get 
\begin{alignat*}3
\frac{\Vol B_0(\lambda_a)}{\det \Lambda}\leq 2^\Da\frac{\lambda_a^\Da}{\lambda_1\cdots\lambda_\Da}\leq 
2^\Da\frac{\lambda_a^{\Da-1}}{\lambda_1^{a-1}\lambda_a^{\Da-a}}\leq 2^D \frac{(\kappa L)^{\Da-1}}{{\lambda_1}^{a-1}{\lambda_a}^{\Da-a}}.
\end{alignat*}
Moreover, 
\begin{alignat*}3
\max_{0\leq i<\Da}\frac{(2\pi\Da\lambda_a)^i}{\lambda_1\cdots\lambda_i} \leq 
(2\pi\Da)^{\Da-1}\frac{\lambda_a^{\Da-1}}{\lambda_1^{a-1}\lambda_a^{\Da-a}}
\leq (2\pi\Da)^{\Da-1} \frac{(\kappa L)^{\Da-1}}{{\lambda_1}^{a-1}{\lambda_a}^{\Da-a}}.
\end{alignat*}

Next suppose
$\kappa L<\lambda_a$.
Then, as $Z\subset B_0(\kappa L)$, we have $|Z\cap\Lambda(a)|=0$. Again, by Minkowski's second Theorem and by $Z\subset B_0(\kappa L)$ we get
\begin{alignat*}1
\frac{\Vol Z}{\det \Lambda}\leq \frac{(2\kappa L)^\Da}{\lambda_1\cdots \lambda_\Da}
\leq 2^\Da\left(\frac{\kappa L}{\lambda_1}\right)^{a-1}\left(\frac{\kappa L}{\lambda_a}\right)^{\Da-a+1}
\leq 2^\Da \frac{(\kappa L)^{\Da-1}}{{\lambda_1}^{a-1}{\lambda_a}^{\Da-a}}.
\end{alignat*}
This completes the proof.
\end{proof}

\section{Counting via flows and partition techniques}\label{toolsinvprin}
Let $r$ and $s$ be non-negative integers not both zero, and put $d=r+2s$ and $q=r+s-1$.
For $1\leq i\leq r+s$ we set $d_i=1$ if $i\leq r$ and $d_i=2$ otherwise. 
We write $\z_i=(z_{i1},\ldots,z_{in})$ for variables in $K_i^n$, where $K_i=\IR$ if $i\leq r$ and $K_i=\IC$ if $i>r$.
Moreover, we write
\begin{alignat*}1
|\z_i|_{\infty}&=\max\{|z_{i1}|,\ldots,|z_{in}|\},\\
|(1,\z_i)|_{\infty}&=\max\{1,|z_{i1}|,\ldots,|z_{in}|\}.
\end{alignat*}
For $T\geq 1$ we define the set
\begin{alignat*}1
\ST=\left\{(\z_1,\ldots,\z_{r+s})\in \prod_{i=1}^{r+s}K_i^n; \prod_{i=1}^{r+s}|(1,\z_i)|_{\infty}^{d_i}\leq T\right\}.
\end{alignat*}
For each subset $I\subset \{1,2,\ldots,r+s\}$ and $I^c=\{1,2,\ldots,r+s\}\backslash I$ we define 
\begin{alignat}1\label{defSI}
\S_I(T)=\{(\z_1,\ldots,\z_{r+s})\in \ST; &|\z_i|_{\infty}\geq 1 \text{ for } i\in I \text{ and }\\
\nonumber &|\z_i|_{\infty}<1 \text{ for } i\in I^c \}.
\end{alignat}
We put 
\begin{alignat*}1
d'=\sum_I d_i, 
\end{alignat*}
and 
\begin{alignat*}1
\q=|I|-1.
\end{alignat*}
Let $\mathcal{T}$ be the group of $\IR$-linear maps $\phi$ on $\prod_{i=1}^{r+s}{K_i^n}$ of the form 
\begin{alignat}1\label{phidef}
\phi(\z_1,\ldots,\z_{r+s})=(\xi_1 \z_1,\ldots,\xi_{r+s}\z_{r+s}) 
\end{alignat}
with positive real $\xi_i$ satisfying 
\begin{alignat}1\label{phidet}
\prod_{i=1}^{r+s}\xi_i^{d_i}=1,
\end{alignat}
so that $\det \phi=1.$
The following theorem is an important intermediate step.
\begin{theorem}\label{mainthm}
Suppose $\q=|I|-1\geq 0$.
Let $\Lambda$ be a lattice in the Euclidean space $\prod_{i=1}^{r+s}{K_i^n}$ and suppose there exist positive real numbers 
$\mun_1,\ldots,\mun_{nd}$ such that
$\lambda_p(\phi(\Lambda))\geq \mun_p$ for $1\leq p\leq nd$ and all $\phi \in \mathcal{T}$. Then,
for $T\geq 1$, one has 
\begin{alignat*}1
\left||\Lambda\cap \S_I(T)|-\frac{\Vol \S_I(T)}{\det \Lambda}\right|
&\leq c_6 (\Log T)^{\q}\max_{0\leq p<nd}\frac{T^{p/d}}{\mun_1\cdots\mun_p},\\
\left||\Lambda\cap \S_I(T)|-\frac{\Vol \S_I(T)}{\det \Lambda}\right|
&\leq c_7 (\Log T)^{\q}\frac{T^{n-1/d}}{\mun_1^{nd-1}},
\end{alignat*}
where $c_6=c_6(n,d)$ and $c_7=c_7(n,d)$ depend only on $n$ and $d$.
For $p=0$ the expression in the maximum is to be understood
as $1$. Moreover, if $T<(\mun_1/\ka)^d$ we have
\begin{alignat*}1
|\Lambda\cap \S_I(T)|=0,
\end{alignat*}
where $\ka=\sqrt{dn}\exp(\sqrt{q})$.
\end{theorem}

\section{Preliminaries}\label{1subsec4}
Unless explicitly mentioned otherwise (which will be the case only in Section \ref{Corollary2}) we always assume $I\neq \emptyset$.
Suppose $I=\{i_1,\ldots,i_p\}$ with $i_1<\cdots<i_p$ then we put $(\z_i)_I=(\z_{i_1},\ldots,\z_{i_p})$.
For subsets $\Ze \subset \prod_I{K_i^n}$ and $\Zz \subset \prod_{I^c}{K_i^n}$ we identify the Cartesian
product $\Ze\times \Zz$ with $\Ze$ if $I^c$ is empty. It is more convenient to 
group the coordinate vectors according to their maximum norm, and thus we redefine
\begin{alignat}3\label{SITcartprod}
\S_I(T)=&\left\{(\z_i)_I\in \prod_I{K_i^n};\prod_I |\z_i|_{\infty}^{d_i}\leq T,|\z_i|_{\infty}\geq 1 \text{ for } i\in I\right\}\\
\nonumber&\times \left\{({\z}_{i})_{I^c}\in \prod_{I^c}{K_i^n};
|{\z}_{i}|_{\infty}<1   \text{ for }i \in I^c\right\}.
\end{alignat}
As we study the cardinality $|\Lambda\cap \S_I(T)|$ we shall permute the coordinates 
of $\Lambda$ in the same manner, and we modify $\phi \in \mathcal{T}$ accordingly to act on $\prod_I{K_i^n}\times\prod_{I^c}{K_i^n}$. Of course, 
this leaves the volume $\Vol \S_I(T)$ and the values $\lambda_i(\phi(\Lambda))$ invariant.
Let $\Sigma$ be the hyperplane in $\IR^{\q+1}$
defined by $x_1+\cdots+x_{\q+1}=0$ and 
\begin{alignat*}1
\vdelta=(d_i/d')_I. 
\end{alignat*}
Let $F$ be a set in $\Sigma$ and put $F(T)$ for the vector sum
\begin{alignat}3
\label{vecsum0}
F(T)=F+\vdelta(-\infty,\log T].
\end{alignat}
The map $(\z_i)_I\longrightarrow (d_i\log|\z_i|_{\infty})_I$
sends $\prod_I{K_i^n}\backslash\{\0\}$ to $\IR^{\q+1}$.
Now we define 
\begin{alignat}3
\label{inkl1}
S_{F}(T)=\left\{(\z_i)_I\in \prod_I{{K_i^n}\backslash\{\0\}};(d_i\log|\z_i|_{\infty})_I \in F(T)\right\}.
\end{alignat}
Directly from the definition we get
\begin{alignat}3
\label{SFThomexp}
&S_{F}(T)=T^{1/d'}S_{F}(1).
\end{alignat}
Moreover, if $F$ lies in a ball centered at zero of radius $r_F$, then for any $(\z_i)_I\in S_{F}(T)$  
\begin{alignat}3
\label{SFnormbound}
|\z_i|_{\infty}\leq \exp(r_F)T^{1/d'} \quad (i\in I).
\end{alignat}
For non-negative reals $a_i$ ($i\in I$) let us write
\begin{alignat}3\label{Esets}
\E((a_i)_I)=\left\{(\z_i)_I\in \prod_I{{K_i^n}};|\z_i|_{\infty}\geq a_i \text{ for }i\in I\right\}.
\end{alignat}

\section{Partitioning and transforming $\S_I(T)$}\label{decomptransSFT}
In Section \ref{proofmainthm} we will prove that for $\q>0$ we have
\begin{alignat*}1
\S_I(T)=\left(S_{\F}(T)\cap \E((1)_I)\right)\times \{({\z}_{i})_{I^c};
|{\z}_{i}|_{\infty}<1   \text{ for }i \in I^c\} 
\end{alignat*}
for a certain $\F\subset \Sigma$.
In this section we focus on the first component $S_{\F}(T)\cap \E((1)_I)$ but we 
will allow arbitrary sets $F\subset \Sigma$.
Throughout this section we assume
\begin{alignat*}3
\q>0.
\end{alignat*}
Fix once and for all an orthonormal basis $e_1,\ldots,e_{\q}$ of $\Sigma\subset \IR^{\q+1}$.
For ${\bf{j}}=(j_1,\ldots,j_{\q})\in \IZ^{\q}$ we define the fundamental cell 
\begin{alignat*}3
\Ci=j_1e_1+[0,1)e_1+\cdots +j_{\q}e_{\q}+[0,1)e_{\q}.
\end{alignat*}
For $F\subset \Sigma$ we define
\begin{alignat*}3
\Di=\Ci\cap F.
\end{alignat*}
Let $\m_{F}$ be the set of those ${\i}$ that satisfy $\Di\neq \emptyset$.
Clearly,
\begin{alignat}3\label{decompositionF}
F=\bigcup_{\m_{F}}\Di,
\end{alignat}
and the latter is a disjoint union.
\begin{lemma}\label{mFbound}
Suppose $F$ is a subset of $\Sigma$ and $F\subset B_0(r_F)$ with $r_F\geq 1$. Then
\begin{alignat*}3
|\m_F|\ll {r_F}^{\q}.
\end{alignat*} 
\end{lemma}
\begin{proof}
Clearly, $F$ lies in the cube $[-r_F,r_F]e_1+\cdots +[-r_F,r_F]e_{\q}$ which has non-empty intersection with at most
$(2\lceil r_F \rceil+1)^{\q}$ fundamental cells $\Ci$ (here $\lceil r_F \rceil$ denotes the smallest integer not smaller 
than $r_F$). Since $r_F\geq 1$ the lemma follows.
\end{proof}
Now (\ref{decompositionF}) leads to 
\begin{alignat}3
\label{partSF1}
S_{F}(T)=\bigcup_{\m_{F}}S_{\Di}(T),
\end{alignat}
which again is a disjoint union.
For each vector
${\bf j}=(j_1,\ldots,j_{\q})\in \IZ^{\q}$ we define a translation $\tr$
on $\IR^{\q+1}$ by 
\begin{alignat*}3
\tr(x)=x-\sum_{p=1}^{\q}j_p e_p=x-u(\i),
\end{alignat*}
where $u(\i)=(u_{i})_I=\sum_{p=1}^{\q}j_p e_p$.
This translation sends $\Sigma$ to $\Sigma$ and
$\Ci$ to $\C0$. For $i\in I$ set $\gamma_i=\gamma_i(\i)=\exp(-u_i/d_i)$, so that $\gamma_i>0$, 
\begin{alignat}3
\label{prodgammai}
\prod_I\gamma_i^{d_i}=1,
\end{alignat}
and 
\begin{alignat*}3
(d_i\log|\gamma_i{\z}_i|_{\infty})_I=\tr((d_i\log|{\z}_i|_{\infty})_I).
\end{alignat*}
Hence, for the automorphism $\ti$ of $\prod_I K_i^n$ defined by
\begin{alignat*}3
\ti({\z}_i)_I=(\gamma_i{\z}_i)_I,
\end{alignat*}
we have
\begin{alignat*}3
\ti S_{F}(T)=S_{\tr(F)}(T).
\end{alignat*}
As $\tr(\Di)=\tr(F)\cap \C0$
we get 
\begin{alignat}3\label{tiSFT}
\ti S_{\Di}(T)=S_{\tr(F)\cap \C0}(T).
\end{alignat}
Moreover, we have
\begin{alignat*}3
\ti \E(1)_I=\left\{({\z}_i)_I\in \prod_I{K_i^n};|\z_i|_{\infty}\geq \gamma_i \text{ for }i\in I\right\}=\E((\gamma_{i})_I).
\end{alignat*}
As $\C0\subset B_0(\sqrt{\q})$ we get from (\ref{SFnormbound}) that
for any $(\z_i)_I\in S_{\C0}(T)$ 
\begin{alignat}3\label{SC0normbound}
|\z_i|_{\infty}\leq \exp(\sqrt{\q})T^{1/d'} \quad (i\in I).
\end{alignat}
We extend $\ti$ to a diagonal endomorphism $\Ti$ on $\prod_I K_i^n\times \prod_{I^c} K_i^n$ by setting 
\begin{alignat}1\label{defTi}
\Ti(((\z_i)_I,(\z_i)_{I^c}))=(\ti(\z_i)_I,(\z_i)_{I^c})=((\gamma_i\z_i)_I,(\z_i)_{I^c}).
\end{alignat}
Next we put
\begin{alignat}3\label{defSiF}
\SiF=\left(S_{\Di}(T)\cap \E((1)_I))\right)\times \{({\z}_{i})_{I^c};
|{\z}_{i}|_{\infty}<1   \text{ for }i \in I^c\}.
\end{alignat}

\section{Further transforming $\S_I(T)$}\label{furthertrans}
We define a map
\begin{alignat}3\label{defpsi}
\psi:\prod_I K_i^n\times \prod_{I^c} K_i^n\longrightarrow \prod_I K_i^n\times \prod_{I^c} K_i^n
\end{alignat}
by 
\begin{alignat*}3
\psi(((\z_i)_I,(\z_i)_{I^c}))=(\psi_1((\z_i)_I),\psi_2((\z_i)_{I^c})),
\end{alignat*}
where
\begin{alignat*}3
\psi_1((\z_i)_I)&=((T^{-1/d'+1/d} \z_i)_I),\\
\psi_2((\z_i)_{I^c})&=((T^{1/d} \z_i)_{I^c}).
\end{alignat*}
For $\q=q$ (i.e., for $I^c=\emptyset$) we interpret, of course, $\psi=\psi_1$ as the
identity on $\prod_I K_i^n=\prod_{i=1}^{r+s} K_i^n$.
As $d'=\sum_I d_i$ we see that 
\begin{alignat}3\label{prodeta}
\det \psi=\prod_I T^{d_in(-1/d'+1/d)}\prod_{I^c} T^{d_in/d}=1.
\end{alignat}
Therefore, $\psi$ lies in
$\mathcal{T}$. 

First suppose $\q=0$, so that $I=\{i\}$ is a singleton. Then
\begin{alignat}3\label{SITq=0}
\psi\S_I(T)=&\\
\nonumber\left\{\z_i\in {K_i^n};T^{-1/d'+1/d}\leq|\z_i|_{\infty}\leq T^{1/d}\right\}
 &\times \left\{({\z}_{i'})_{i'\neq i} \in \prod_{i'\neq i}{K_{i'}}^n;
|{\z}_{i'}|_{\infty}<T^{1/d}\text{ for }i'\neq i\right\}.
\end{alignat}
Now suppose $\q>0$. For $\i\in \IZ^{\q}$ we set
\begin{alignat}1\label{defC1}
\Ze&=\psi_1\left(\ti S_{\Di}(T)\cap \ti\E((1)_I)\right)\subset \prod_I K_i^n,
\end{alignat}
and, with $\Ti$ as in (\ref{defTi}), we define
\begin{alignat}1
\label{defpsiTi}
\psiTi=\psi\circ\Ti.
\end{alignat}
Moreover, we set
\begin{alignat}1
\label{defC2} \Zz=\psi_2\left\{({\z}_{i})_{I^c};
|{\z}_{i}|_{\infty}<1\text{ for }i \in I^c\right\}=\left\{({\z}_{i})_{I^c};
|{\z}_{i}|_{\infty}<T^{1/d}\text{ for }i \in I^c\right\}\subset \prod_{I^c} K_i^n,
\end{alignat}
so that  
\begin{alignat*}1
\Ze\times \Zz=\psiTi \SiF. 
\end{alignat*}

\begin{lemma}\label{superballlemma}
Let $\ka=\sqrt{dn}\exp(\sqrt{q})$ be as in Theorem \ref{mainthm}.
If $\q=0$ then we have
\begin{alignat}1\label{ballq0}
\psi \S_I(T) \subset B_0(\ka T^{1/d}).
\end{alignat}
If $\q>0$ and  $\i\in \IZ^{\q}$ then we have 
\begin{alignat}1\label{superball}
\psiTi \SiF\subset B_0(\ka T^{1/d}).
\end{alignat}
In particular, 
\begin{alignat}1\label{superballCi}
\Zi&\subset B_0(\ka T^{1/d}) \quad (1\leq p\leq 2)
\end{alignat}
for the respective balls $B_0(\ka T^{1/d})$.
\end{lemma}
\begin {proof}
As $\ka\geq \sqrt{(q+1)n}$ the claim (\ref{ballq0}) follows immediately from (\ref{SITq=0}).
Next suppose $\q>0$.
Recall from (\ref{tiSFT}) that $\ti S_{\Di}(T)\subset S_{\C0}(T)$. From (\ref{SC0normbound}), and not forgetting the effect of $\psi_1$, 
we see that for any $(\z_i)_I$ in
$\Ze$ we have $|\z_i|_\infty\leq \exp(\sqrt{\q})T^{1/d}$ ($i\in I$).
And, obviously, we also have $|\z_i|_\infty\leq \exp(\sqrt{\q})T^{1/d}$ ($i\in I^c$) for any $(\z_i)_{I^c}$ in $\Zz$.
This proves (\ref{superball}). 
\end{proof}

\section{Lipschitz parameterizations}\label{SectionLippara}

In this section we shall prove that the sets $\psi \S_I(T)$ (if $\q=0$), and $\psiTi \SiF$ (if $\q>0$)
have Lipschitz parameterizable boundaries with Lipschitz constant $L\ll T^{1/d}$. To this end we need a few simple lemmas.
For $\q>0$ we will identify $\Sigma$ with $\IR^{\q}$ via the basis $e_1,\ldots,e_{\q}$ from Section \ref{decomptransSFT}. 
For a subset $\mathcal{Z}$ of Euclidean space we write $\partial \mathcal{Z}$ for its topological boundary.
\begin{lemma}\label{LipF}
Suppose ${\q}>0$, and let $F$ be a set in $\Sigma$ such that
$\partial F$ is in Lip$({\q},\M',L')$, and, moreover, assume 
$F$ lies in $B_0(r_F)$. Then 
$\partial S_{F}(1)$ is in Lip$(d'n,\widetilde{\M},\widetilde{L})$
with $\widetilde{\M}$ and $\widetilde{L}$ depending only on 
$n, \q, \M', L', r_F$.
\end{lemma}
\begin{proof}
The case $n>1$ follows directly from \cite[Lemma 3]{1} (see also \cite[Lemma 7.1]{art1} for a more detailed and completely explicit version). 
However, for $n=1$ the proof remains correct without change. 
\end{proof}

\begin{lemma}\label{lippar}
Suppose $\q>0$, and
recall the definition of $\tr$ and $\Di$ from Section \ref{decomptransSFT}.
Let $Y\geq 1$ be a real number and suppose the boundary of $\tr\Di$ lies in Lip$(\q,M',L')$ with $M'\ll 1$ and $L'\ll 1$.
Then the boundary of $\ti S_{\Di}(Y)$ lies in Lip$(d'n,M,L)$ with $M\ll 1$ and $L\ll Y^{1/d'}$.
\end{lemma}  
\begin{proof}
Clearly, $\tr(\Di)=\tr(F)\cap \C0$ is contained in $B_0(\sqrt{\q})$.
Now $\ti S_{\Di}(Y)=S_{\tr(\Di)}(Y)$ and thus the lemma follows from (\ref{SFThomexp}) and Lemma \ref{LipF}.
\end{proof}

\begin{lemma}\label{Lippsi}
If $\q=0$ then $\partial \psi \S_I(T)$ lies in Lip$(dn,M,L)$ with $M\ll 1$ and $L\ll T^{1/d}$.
If $\q>0$ and $\partial \tr\Di$ lies in Lip$(\q,M',L')$ with $M'\ll 1$ and $L'\ll 1$ then
the set $\partial \psiTi \SiF$ lies in Lip$(dn,M,L)$ with $M\ll 1$ and $L\ll T^{1/d}$.\\
\end{lemma}
\begin{proof}
Fist suppose $\q=0$. The sets in $K_i^n$ defined by $|\z_i|_\infty=\zeta$ are in Lip$(d_i n,2n,\zeta')$ with $\zeta'\ll \zeta$, 
e.g., we can take $2n$ linear (if $i\leq r$) 
or $n$ trigonometrical (if $i>r$) maps. Then one easily gets a parameterization of the sets $|\z_i|_\infty=\zeta_1$, $|\z_{i'}|_\infty\leq \zeta_2$ ($i'\neq i$)
in $\prod_{I} K_i^n\times \prod_{I^c} K_i^n$ 
with $M\ll 1$ maps and Lipschitz constants $L\ll \max\{\zeta_1,\zeta_2\}$. In view of (\ref{SITq=0}) this proves the lemma for $\q=0$.

Now suppose $\q>0$. We need to show that $\partial(\Ze\times \Zz)$
lies in Lip$(dn,M,L)$.
Clearly, $\partial(\Ze\times \Zz)$ is contained in the union of $\overline{\Ze}\times \partial \Zz$ and $\partial \Ze\times \overline{\Zz}$,
where the bar denotes the topological closure.
Moreover, by (\ref{superballCi}) we know $\overline{\Ze}$ and $\overline{\Zz}$ lie both in a ball $B_0(\ka T^{1/d})$.
Therefore, it suffices to show that $\partial \Ze\in$ Lip$(d'n,M'',L'')$ and, if $d-d'>0$, also
$\partial \Zz\in$ Lip$((d-d')n,M'',L'')$ with some $M''\ll 1$ and some $L''\ll T^{1/d}$.
Next note that
\begin{alignat*}1
&\psi_1\ti(S_{\Di}(T))=T^{1/d}S_{\tr\Di}(1),\\
&\psi_1\ti\left(\E((1)_I)\right)=\E((T^{1/d-1/d'}\gamma_{i})_I).
\end{alignat*}
As $\Ze$ is the intersection of these two sets, we see that $\partial \Ze$ is covered by the union of 
$\partial \E((T^{1/d-1/d'}\gamma_{i})_I)\cap\overline{\Ze}$ and $\partial T^{1/d}S_{\tr\Di}(1)$. Regarding the latter recall that 
$\tr\Di\subset\C0\subset B_0(\sqrt{\q})$ and $\partial\tr \Di$ lies in Lip$(\q,M',L')$.
Therefore, we can apply Lemma \ref{LipF} to conclude $\partial T^{1/d}S_{\tr\Di}(1)$ lies in Lip$(d'n,M'',L'')$
with some $M''\ll 1$ and some $L''\ll T^{1/d}$.
And for $\partial \E((T^{1/d-1/d'}\gamma_{i})_I)\cap\overline{\Ze}$ we use the same argument as for $\q=0$ to see that it is in 
Lip$(d'n,M'',L'')$ with an $M''\ll 1$ and an $L''\ll T^{1/d}$. And again, the same argument shows that, for $d>d'$, 
$\partial \Zz$ lies in Lip$((d-d')n,M'',L'')$ with an $M''\ll 1$ and an $L''\ll T^{1/d}$. This proves the 
Lemma \ref{Lippsi}.
\end{proof}

\section{Proof of Theorem \ref{mainthm}}\label{proofmainthm}
To simplify the notation we write $\S_I$ for $\S_I(T)$.
First we assume $\q=0$.\\
Recall that $\psi$ lies in
$\mathcal{T}$, and, clearly, we have $|\S_I\cap \Lambda|=|\psi \S_I\cap \psi\Lambda|$. By (\ref{ballq0}) we have 
$\psi(\S_I)\subset B_0(\ka T^{1/d})$, and by hypothesis of Theorem \ref{mainthm} we have 
$\lambda_i(\psi\Lambda)\geq \mun_i$ for $1\leq i\leq dn$.
Thanks to Lemma \ref{Lippsi} we can apply Theorem \ref{TWi1} which gives the first inequality of Theorem \ref{mainthm}.
For the second inequality we apply Corollary \ref{TWi1cor} with $a=1$ and note that $\0\notin \psi(\S_I)$.
And finally, as $\0\notin \psi(\S_I)$ and $\psi(\S_I)\subset B_0(\ka T^{1/d})$ we see that 
$|\Lambda\cap \S_I|=0$ if $T^{1/d}<(1/\ka)\mun_1$. 
This finishes the proof of Theorem \ref{mainthm} for $\q=0$.\\

For the rest of this section we assume $\q>0$, and, for the rest of the paper, we fix $F$ as 
\begin{alignat}3\label{defF0}
\F=(\IR_{\geq 0}^{\q+1}-\vdelta\log T)\cap \Sigma. 
\end{alignat}
\begin{lemma}\label{lemmaSIF}
We have
\begin{alignat*}3
\S_I=\left(S_{\F}(T)\cap \E((1)_I)\right)\times \{({\z}_{i})_{I^c};
|{\z}_{i}|_{\infty}<1   \text{ for }i \in I^c\}.
\end{alignat*}
\end{lemma}
\begin{proof}
In view of (\ref{SITcartprod}) it suffices to show  
\begin{alignat}3\label{firstcompeq}
\left\{(\z_i)_I\in \prod_I{K_i^n};\prod_I |\z_i|_{\infty}^{d_i}\leq T,|\z_i|_{\infty}\geq 1 \text{ for } i\in I\right\}=
S_{\F}(T)\cap \E((1)_I)
\end{alignat}
From the definitions (\ref{inkl1}) and (\ref{Esets}) we see immediately that the right hand-side is contained in the left hand-side for any 
choice of $\F\subset\Sigma$ whatsoever. Now for the other inclusion note 
that the left hand-side in (\ref{firstcompeq}) means
\begin{alignat*}3
(d_i\log|\z_i|_\infty)_I\in \IR_{\geq 0}^{\q+1}\cap \left(\Sigma+\vdelta(-\infty,\log T]\right).
\end{alignat*}
Thus we need to show
\begin{alignat*}3
\IR_{\geq 0}^{\q+1}\cap \left(\Sigma+\vdelta(-\infty,\log T]\right)
\subset \F(T)=\left(\left(\IR_{\geq 0}^{\q+1}-\vdelta\log T\right)\cap\Sigma\right)+\vdelta(-\infty,\log T].
\end{alignat*}
Any element in the set on the left hand-side can be written as $\vx+\vdelta t$ with $\vx\in\Sigma$ and $t\in(-\infty, \log T]$.
As $\vx+\vdelta t\in \IR_{\geq 0}^{\q+1}$ we get $\vx\in \IR_{\geq 0}^{\q+1}-\vdelta\log T\cap\Sigma$, and therefore
\begin{alignat*}3
\vx+\vdelta t\in \left(\left(\IR_{\geq 0}^{\q+1}-\vdelta\log T\right)\cap\Sigma\right)+\vdelta(-\infty,\log T].
\end{alignat*}
This concludes the proof.
\end{proof}

\begin{lemma}
We have
\begin{alignat}3\label{bFbound}
\F\subset B_0(2\log T) 
\end{alignat}
\end{lemma}
\begin{proof}
Suppose $(x_1,\ldots,x_{\q+1})\in \F$. As $x_1+\cdots +x_{\q+1}=0$ we see that the sum over the positive coordinates
equals minus the sum over the negative coordinates and thus $|x_1|+\cdots +|x_{\q+1}|\leq 2\sum_I (d_i/d')\log T=2\log T$.
This proves the lemma. 
\end{proof}
Recall the definition of $\SibF$ from (\ref{defSiF}). The disjoint union (\ref{partSF1}), in conjunction with Lemma \ref{lemmaSIF},
leads to the disjoint union 
\begin{alignat}3\label{ZIpart}
\S_I=\bigcup_{\m_{\F}} \SibF,
\end{alignat}
which in turn yields
\begin{alignat*}5
|\S_I\cap \Lambda|=\sum_{\m_{\F}}|\SibF\cap \Lambda|.
\end{alignat*}
As the $\psiTi$ are automorphisms we conclude
\begin{alignat}5
\label{partSF4}
|\S_I\cap \Lambda|=\sum_{\m_{\F}}|\psiTi\SibF\cap \psiTi\Lambda|.
\end{alignat}
We will apply Lemma \ref{Lippsi} with our choice of $\F$ given in (\ref{defF0}). 
We start off by verifying the necessary conditions.
\begin{lemma}\label{SF0TST}
Let $\F$ be as in (\ref{defF0}). There exist $M'\ll 1$ and $L'\ll 1$ such that $\partial \tr\F_{\i}$ lies in Lip$(\q,M',L')$.
\end{lemma}
\begin{proof}
Clearly, $\F$, and therefore also $\tr\F$, is convex. And, clearly, $\C0$ is convex and contained in $B_0(\sqrt{\q})$.
Hence $\tr\F_{\i}=\tr\F\cap \C0$ is convex and lies in $B_0(\sqrt{\q})$.
Now if $\q=1$ the lemma is trivial, and if $\q>1$ it follows immediately from 
\cite[Theorem 2.6]{WidmerLNCL}.
\end{proof}

\begin{lemma}\label{LipparpsiTiSibF}
The set $\partial \psiTi(\SibF)$ lies in
Lip$(dn,M,L)$ with some $M\ll 1$ and some $L\ll T^{1/d}$.
\end{lemma}
\begin{proof}
This is an immediate consequence of Lemma \ref{SF0TST} and Lemma \ref{Lippsi}. 
\end{proof}

\begin{lemma}\label{latticeptsest}
We have
\begin{alignat*}1
\left||\psiTi(\SibF)\cap \psiTi(\Lambda)|-\frac{\Vol \SibF}{\det\Lambda}\right|&\ll \max_{0\leq p<dn}\frac{T^{p/d}}{\mun_1\cdots\mun_p},\\
\left||\psiTi(\SibF)\cap \psiTi(\Lambda)|-\frac{\Vol \SibF}{\det\Lambda}\right|&\ll \frac{T^{n-1/d}}{\mun_1^{nd-1}},\\
|\psiTi(\SibF)\cap \psiTi(\Lambda)|&=0 \text{ if }T^{1/d}<(1/\ka)\mun_1.
\end{alignat*}
\end{lemma}
\begin{proof}
Again, we want to apply Theorem \ref{TWi1} and Corollary \ref{TWi1cor}. First recall that $\psiTi \in \mathcal{T}$, in particular,
$\Vol \psiTi\SibF=\Vol \SibF$ and $\det\psiTi(\Lambda)=\det(\Lambda)$.
By Lemma \ref{LipparpsiTiSibF} we know $\partial \psiTi(\SibF)$ lies in
Lip$(dn,M,L)$ with some $M\ll 1$ and some $L\ll T^{1/d}$.
By (\ref{superball}) we have $\psiTi\SibF \subset B_0(\kappa T^{1/d})$ with $1\leq \kappa \ll 1$, and as $\0\notin \S_I$ we also have
$\0\notin \psiTi \SibF$.
Applying Theorem \ref{TWi1cor} and Corollary \ref{TWi1cor}, and using the hypothesis 
$\lambda_p(\psiTi(\Lambda))\geq \eta_p$ yields the inequalities of the lemma.
And the last statement follows just as in the case $\q=0$. 
\end{proof}
\begin{lemma}\label{m-est}
We have 
\begin{alignat*}1
|\m_{\F}|\ll (\Log T)^{\q}.
\end{alignat*}
\end{lemma}
\begin{proof}
This follows immediately from (\ref{bFbound}) and Lemma \ref{mFbound}.
\end{proof}
We can now easily conclude the proof of Theorem \ref{mainthm}.
Combining (\ref{partSF4}) and Lemma \ref{latticeptsest} with (\ref{ZIpart}) implies
\begin{alignat*}3
\left||\S_I\cap \Lambda|-\frac{\Vol \S_I}{\det\Lambda}\right|&\ll \sum_{\m_{\F}}\max_{0\leq p<dn}\frac{T^{p/d}}{\mun_1\cdots\mun_p},\\
\left||\S_I\cap \Lambda|-\frac{\Vol \S_I}{\det\Lambda}\right|&\ll \sum_{\m_{\F}}\frac{T^{n-1/d}}{\mun_1^{nd-1}}.
\end{alignat*}
And, if  $T^{1/d}<(1/\ka)\mun_1$, we have
\begin{alignat*}3
|\S_I\cap \Lambda|&=\sum_{\m_{\F}}0=0.
\end{alignat*}
Finally, we use Lemma \ref{m-est} to deduce
\begin{alignat*}1
\sum_{\m_{\F}}1\ll (\Log T)^{\q}.
\end{alignat*}
This proves Theorem \ref{mainthm}.

\section{Estimates for the successive minima}\label{estsuccmin}
In this section, we state the fact that the successive minima of the lattice $\phi\sigma\Oseen^n$ are bounded away from zero,
uniformly in $\phi \in \mathcal{T}$. We also state a crucial refinement involving a critical higher successive minimum
$\lambda_l$ and two other results.
All these results  are slight generalizations of those in 
\cite[Section 9]{art1} but they are proved by exactly the same arguments.
Therefore we skip the proofs  and simply state the lemmas. 

As in Section \ref{generalisation} let $K/k$ be an extension of number fields, and $d=[K:\IQ]$.
Recall that $\sigma_1,\ldots,\sigma_d$ denote 
the embeddings from $K$ to $K_i$, ordered such that
$\sigma_{r+s+i}=\overline{\sigma}_{r+i}$ for 
$1\leq i \leq s$.
We write
\begin{alignat}3\label{embedd1dim}
&\sigma:K\longrightarrow \prod_{i=1}^{r+s}{K_i}\\
\nonumber&\sigma(\alpha)=(\sigma_1(\alpha),\ldots,\sigma_{r+s}(\alpha)).
\end{alignat}
Let $\phi$ be as in (\ref{phidef}).
By abuse of notation we may regard $\phi$ also as an automorphism of $\IR^r\times \IC^s$, and 
from now on, depending on the argument, we view $\phi$ as an automorphism of $\IR^r\times \IC^s$ or $\IR^{rn}\times \IC^{sn}$. 
Applying $\phi$ to the lattice $\sigma\Oseen$ gives a new lattice  
$\phi\sigma\Oseen$ in $\IR^r\times \IC^s$.
As is well-known, see, e.g., \cite[Chapter VIII, Lemma 1]{18}, we can choose linearly independent 
vectors 
\begin{alignat*}3
v_1=\phi\sigma(\theta_1),\ldots,v_d=\phi\sigma(\theta_d)
\end{alignat*}
of the lattice $\phi\sigma\Oseen$ with 
\begin{alignat}3
\label{visuccmin}
|v_i|&=\lambda_i(\phi\sigma\Oseen) \qquad (1\leq i\leq d)
\end{alignat}
for the successive minima $\lambda_i(\phi\sigma\Oseen)$.
The $v_1,\ldots,v_d$ are $\IR$-linearly independent. Hence, $\theta_1,\ldots,\theta_d$ are $\IQ$-linearly 
independent, and therefore $\frac{\theta_1}{\theta_1},\ldots,\frac{\theta_d}{\theta_1}$
are $\IQ$-linearly independent. As $[K:\IQ]=d$ we get  $K=\IQ(\frac{\theta_1}{\theta_1},\ldots,\frac{\theta_d}{\theta_1})
=k(\frac{\theta_1}{\theta_1},\ldots,\frac{\theta_d}{\theta_1})$,
and this allows the following definition.
\begin{definition}\label{ldef}
Let  $l \in \{1,\ldots,d\}$ be minimal with 
$K=k(\frac{\theta_1}{\theta_1},\ldots,\frac{\theta_{l}}{\theta_1})$.
\end{definition}
We abbreviate 
\begin{alignat}3\label{lambdai}
\lambda_i=\lambda_i(\phi\sigma\Oseen)
\end{alignat}
for $1\leq i\leq d$. Recall the definition of $\eO$ from (\ref{defeO}).
\begin{lemma}\label{minest1}
We have
\begin{alignat*}3
\lambda_1 &\geq \sqrt{d/2}\eO.
\end{alignat*}
\end{lemma} 
\begin{lemma}\label{lemma2.7}
With $K_0=k(\frac{\theta_1}{\theta_1},\ldots,\frac{\theta_{l-1}}{\theta_1})$ if $l\geq 2$ and $K_0=k$ if $l=1$, and $g=[K_0:k]\in G(K/k)$ we have 
\begin{alignat*}3
\lambda_l &\geq \frac{1}{\sqrt{2}ed}\eO\delta_g(K/k).
\end{alignat*}
\end{lemma} 
For the rest of this section we assume that
\begin{alignat*}3
n>1.
\end{alignat*}
\begin{lemma}\label{lemma2.5}
Let $(\omega_1,\ldots,\omega_n)$ be in $\Oseen^n\backslash\{\0\}$ 
with $k(\ldots,\omega_i/\omega_j,\ldots)=K$. Then for  $v=(\phi\sigma\omega_1,\ldots,\phi\sigma\omega_n)$
we have
\begin{alignat*}3
|v| \geq \lambda_l.
\end{alignat*}
\end{lemma}

We remind the reader that $[K:k]=e$, $[k:\IQ]=m$, and $d=em$.

\begin{lemma}\label{lemma2.6}
If $l\geq 2$ then 
\begin{alignat*}3
\frac{l-1}{m}\leq [k\left(\frac{\theta_1}{\theta_1},\ldots,\frac{\theta_{l-1}}{\theta_1}\right):k]
\leq \max\{1,e/2\}.
\end{alignat*}
\end{lemma}

\section{Upper bounds for the projectively non-primitive points}\label{nonprimpts}
We extend the embeddings $\sigma_i$ from (\ref{embedd1dim}) componentwise to get an embedding of $K^n$
\begin{alignat*}3
\sigma:K^n\longrightarrow \prod_{i=1}^{r+s}{K_i^n}.
\end{alignat*}
Depending on the argument we either see $\sigma$ as a map on $K$ or on $K^n$.
Again, let $\phi$ be as in (\ref{phidef}).
In this section we prove an upper bound for the number of nonzero points in $\phi\sigma\Oseen^n$ that (as projective points) do not generate $K/k$
and lie in some ball. For brevity we write
\begin{alignat*}1
\Lambda'=\phi\sigma\Oseen^n\backslash(\phi\sigma\Oseen^n(K/k)\cup \{\0\}).
\end{alignat*}
\begin{lemma}\label{nonprojprimpts}
Suppose $n>1$,
let $B_0(R)$ be the zero centered ball in the Euclidean space $\IR^{nr}\times \IC^{ns}$ of radius $R$, and let $\lambda_i$ be as in (\ref{lambdai}). Then
\begin{alignat*}1
|\Lambda'\cap B_0(R)|\ll \max_{0\leq i\leq d}\frac{R^i}{\lambda_1\cdots\lambda_i}\left(\max_{0\leq i<d}\frac{R^i}{\lambda_1\cdots\lambda_i}\right)^{n-1}.
\end{alignat*}
\end{lemma}
\begin{proof}
We follow the lines of proof in \cite[Proposition 10.1]{art1}.
For $(\phi\sigma\omega_1,\ldots,\phi\sigma\omega_n)$ in $\Lambda'$
the field $k(\ldots,\omega_i/\omega_j,\ldots)$ lies in a strict subfield, say $K_1$, of $K$.
Hence, there exist two different embeddings 
$\sigma_a, \sigma_b$ of $K$ with 
\begin{alignat*}3
\sigma_a\alpha=\sigma_b\alpha
\end{alignat*}
for all $\alpha$ in $K_1$.
Now $(\phi\sigma\omega_1,\ldots,\phi\sigma\omega_n)\neq \0$, and
thus, at least one of the numbers $\omega_1,\ldots,\omega_n$ is nonzero.
By symmetry we lose only a factor $n$ if we assume 
$\omega_1 \neq 0$. So let us temporarily regard $\omega_1 \neq 0$ as fixed;
then for $2\leq j \leq n$ every $\omega_j$ satisfies
\begin{alignat*}3
\sigma_a\frac{\omega_j}{\omega_1}=\sigma_b\frac{\omega_j}{\omega_1}.
\end{alignat*}
Therefore, all these $\sigma \omega_j$ lie 
in a hyperplane $\P(\omega_1)$ of $\IR^d$, and so all these $\phi\sigma \omega_j$ lie in the hyperplane $\phi\P(\omega_1)$.
As $(\phi\sigma\omega_1,\ldots,\phi\sigma\omega_n)\in B_0(R)$ we have $|\phi\sigma \omega_j|\leq R$.
The intersection of a ball with radius $R$ 
and a hyperplane in $\IR^d$ is a ball in some
$\IR^{d-1}$ with radius $R'\leq R$ and thus, lies in a cube of edge length $2R$. 
Thus, this set belongs to the class Lip$(d,1,2R)$.
Moreover, its $d$-dimensional volume is zero. Hence, by
Theorem \ref{TWi1cor} we obtain the upper
bound
\begin{alignat*}3
\ll \max_{0\leq i<d}\frac{R^i}{\lambda_1\cdots\lambda_i}
\end{alignat*}
for the number of $\phi\sigma\omega_j$ for each $j$ satisfying $2\leq j\leq n$.\\

Next we have to estimate the number of $\phi\sigma\omega_1$. Again, we 
have $|\phi\sigma \omega_1|\leq R$.
Now by virtue of Theorem \ref{TWi1cor} we
deduce the following upper bound
\begin{alignat*}3
\ll \frac{R^d}{\det \phi\sigma\Oseen}+
\max_{0\leq i < d}\frac{R^i}{\lambda_1\cdots\lambda_i}
\end{alignat*}
for the number of  $\phi\sigma\omega_1$.
Going right up to the last minimum, we see that this is bounded by
\begin{alignat*}3
\ll \max_{0\leq i\leq d}\frac{R^i}{\lambda_1\cdots\lambda_i}.
\end{alignat*}
Multiplying the bounds for the number of $\phi\sigma\omega_1$ and 
$\phi\sigma\omega_j$, and then summing over all (of the at most $2^d$) strict subfields $K_1$ of $K$
leads to  
\begin{alignat*}3
|\Lambda'|\ll \max_{0\leq i\leq d}\frac{R^i}{\lambda_1\cdots\lambda_i}\left(\max_{0\leq i<d}\frac{R^i}{\lambda_1\cdots\lambda_i}\right)^{n-1}.
\end{alignat*}
This completes the proof.
\end{proof}

\section{Counting projectively primitive points}\label{estproprimpts}
The height of an element $\valpha=(\alpha_1,\ldots,\alpha_n)\in \Oseen^n\subset \Oseen_K^n$ is given by 
\begin{alignat*}1
H(\valpha)=\prod_{i=1}^{r+s}|(1,\sigma_i(\valpha))|_{\infty}^{d_i/d}.
\end{alignat*}
Therefore, and by the definition (\ref{defSI}) of $\S_I(X^d)$, we have
\begin{alignat}1\label{Naslatticepts}
N(\Oseen_I^n(K/k),X)=|\S_I(X^d)\cap \sigma\Oseen^n(K/k)|.
\end{alignat}
Recall the definitions of $\SiF$, $\psi$, $\psiTi$ and $\F$ from (\ref{defSiF}), (\ref{defpsi}), (\ref{defpsiTi}) and 
(\ref{defF0}). Also recall that $\q=|I|-1$. We permute the coordinates of $\sigma\Oseen^n$ and $\sigma\Oseen^n(K/k)$
as in (\ref{SITcartprod}), so that they become subsets of $\prod_{I}K_i^n\times \prod_{I^c}K_i^n$.
Just as in (\ref{partSF4}) we conclude
\begin{alignat}1\label{primptsdecom}
|\sigma\Oseen^n(K/k)\cap \S_I(T)|=
 \begin{cases}
   |\psi \S_I(T)\cap \psi\sigma\Oseen^n(K/k)|     & \text{if } \q = 0 \\
   \sum_{\m_{\F}}|\psiTi\SibF\cap \psiTi\sigma\Oseen^n(K/k)| & \text{if } \q > 0
  \end{cases}.
\end{alignat}
Of course, the first equation in (\ref{primptsdecom}) holds always, although we use it only for $\q=0$.
It is well known that $\sigma\Oseen^n$ is a lattice of determinant
\begin{alignat*}3
\det\sigma\Oseen^n=(2^{-s}\sqrt{|\Delta_K|}[\Oseen_K:\Oseen])^n.
\end{alignat*}

\begin{proposition}\label{mainprop}
Suppose $T\geq 1$ and $n>1$, and recall that $l$ was defined in Definition \ref{ldef} (Section \ref{estsuccmin}). If $\q=0$ then we have
\begin{alignat*}1
\left| |\psi\sigma\Oseen^n(K/k)\cap \psi \S_I(T)|  -\frac{2^{s_Kn}\Vol \S_I(T)}{(\sqrt{|\Delta_K|}[\Oseen_K:\Oseen])^n}\right|
\ll \frac{T^{n-1/d}}{\lambda_1^{n(l-1)}\lambda_l^{n(d-l+1)-1}},
\end{alignat*}
where $\lambda_i=\lambda_i(\psi\sigma\Oseen)$. If $\q>0$ then we have
\begin{alignat*}1
\left||\psiTi\sigma\Oseen^n(K/k)\cap \psiTi\SibF|-\frac{2^{s_Kn}\Vol \SibF}{(\sqrt{|\Delta_K|}[\Oseen_K:\Oseen])^n}\right|
\ll \frac{T^{n-1/d}}{\lambda_1^{n(l-1)}\lambda_l^{n(d-l+1)-1}},
\end{alignat*}
where $\lambda_i=\lambda_i(\psiTi\sigma\Oseen)$.
\end{proposition}
\begin{proof}
As the case $\q=0$ can be proven by exactly the same arguments we restrict ourselves to the case $\q>0$. 
Let us write $R=\ka T^{1/d}$, where $\ka$ is as in Lemma \ref{superballlemma}, and thus $R\ll T^{1/d}$, and
\begin{alignat*}1
\psiTi\SibF\subset B_0(R).
\end{alignat*}
Put $\Lambda=\psiTi\sigma\Oseen^n$, and recall that $\psiTi\in \mathcal{T}$. The proof splits in two cases.
First we assume
\begin{alignat*}1
R<\lambda_l.
\end{alignat*}
By Lemma \ref{lemma2.5}, and recalling the definition (\ref{Lambda_a}), we conclude $\psiTi\sigma\Oseen^n(K/k)\subset \Lambda(l)$. 
As $\psiTi\SibF\subset B_0(R)$ we get in particular 
$0=|\Lambda(l)\cap \psiTi\SibF|=|\psiTi\sigma\Oseen^n(K/k)\cap \psiTi\SibF|$. Using Lemma \ref{LipparpsiTiSibF}, $\det \psiTi=1$,
and applying Corollary \ref{TWi1cor} proves the proposition in the first case.
Now we assume
\begin{alignat*}1
R\geq \lambda_l.
\end{alignat*}
First we ignore the primitivity condition defining $\Oseen^n(K/k)$ and we count all points in 
$\Lambda(l) \supset \psiTi\sigma\Oseen^n(K/k)$.
Again, using Lemma \ref{LipparpsiTiSibF} and applying Corollary \ref{TWi1cor} yields
\begin{alignat*}1
\left|| \Lambda(l)\cap \psiTi(\SibF)|-\frac{2^{s_Kn}\Vol(\SibF)}{(\sqrt{|\Delta_K|}[\Oseen_K:\Oseen])^n}\right|\ll 
\frac{T^{n-1/d}}{\lambda_1^{n(l-1)}\lambda_l^{n(d-l+1)-1}}.
\end{alignat*}
Next we estimate the number of points in $\Lambda(l)\cap\psiTi(\SibF)$ that do not generate $K/k$ (in the projective sense),
i.e., that do not lie in $\psiTi\sigma\Oseen^n(K/k)$. To this end we apply Lemma \ref{nonprojprimpts}. 
Using $R\geq \lambda_l$ we get the following upper bound for these
\begin{alignat*}1
\ll \max_{0\leq i\leq d}\frac{R^i}{\lambda_1\cdots\lambda_i}\left(\max_{0\leq i<d}\frac{R^i}{\lambda_1\cdots\lambda_i}\right)^{n-1}
\leq \frac{R^{d}}{\lambda_1^{l-1}\lambda_l^{d-l+1}}\left(\frac{R^{d-1}}{\lambda_1^{l-1}\lambda_l^{d-l}}\right)^{n-1}.
\end{alignat*}
As $n>1$ we see that the latter is
\begin{alignat*}1
\leq \frac{R^{dn-1}}{\lambda_1^{n(l-1)}\lambda_l^{n(d-l+1)-1}}\ll \frac{T^{n-1/d}}{\lambda_1^{n(l-1)}\lambda_l^{n(d-l+1)-1}}.
\end{alignat*}
This concludes the proof of the proposition.
\end{proof}

Recall the definitions of $\eO$ and $\mug$ from (\ref{defeO}) and (\ref{defmug}) respectively.
\begin{lemma}\label{lemmamainest}
Suppose $X\geq 1$ and $n>1$. If $\q=0$ then
\begin{alignat*}1
\left||\psi\sigma\Oseen^n(K/k)\cap \psi \S_I(T)| -\frac{2^{s_Kn}\Vol \S_I(T)}{(\sqrt{|\Delta_K|}[\Oseen_K:\Oseen])^n}\right|
\ll \sum_{g\in G(K/k)}\frac{T^{n-1/d}}{\eO^{dn-1}\delta_g(K/k)^{\mug}}.
\end{alignat*}
If $\q>0$ then
\begin{alignat*}1
\left||\psiTi\sigma\Oseen^n(K/k)\cap \psiTi\SibF|-\frac{2^{s_Kn}\Vol \SibF}{(\sqrt{|\Delta_K|}[\Oseen_K:\Oseen])^n}\right|
\ll \sum_{g\in G(K/k)}\frac{T^{n-1/d}}{\eO^{dn-1}\delta_g(K/k)^{\mug}}.
\end{alignat*}
\end{lemma}
\begin{proof}
Recall that $\psi$ and $\psiTi$ are in $\mathcal{T}$, and thus, to estimate the successive minima 
we can apply the results from Section \ref{estsuccmin} with $\phi=\psiTi$ 
and $\phi=\psi$ respectively.
Let $K_0=k(\frac{\theta_1}{\theta_1},\ldots,\frac{\theta_{l-1}}{\theta_1})$ if $l\geq 2$, and let $K_0=k$ if $l=1$, and put 
$g=[K_0:k]$.
In particular, we have $g\in G(K/k)$. Therefore, and by Proposition \ref{mainprop}, it suffices to show 
\begin{alignat}1\label{lamdeest}
{\lambda_1^{n(l-1)}\lambda_l^{n(d-l+1)-1}}\gg \eO^{dn-1}\delta_g(K/k)^{\mug}.
\end{alignat}
First suppose $l=l(\phi)\geq 2$. Then by Lemma \ref{lemma2.6} we have $n(d-l+1)-1\geq \mug$, and thus, (\ref{lamdeest}) follows
immediately from Lemma \ref{lemma2.7}.
Now suppose $l=1$. Then $\delta_g(K/k)=1$ and thus, (\ref{lamdeest}) follows again from Lemma \ref{lemma2.7}.
This proves the lemma.
\end{proof}

\section{Proof of Theorem \ref{mainthm2}}\label{Proof of Theorem2}
We start with the case $n=1$. Hence, by hypothesis, we have $k=K$. From (\ref{Naslatticepts}) and since $\0\notin \S_I(X^d)$ we obtain
\begin{alignat*}1
N(\Oseen_I(K/K),X)=|\sigma\Oseen(K/K)\cap \S_I(X^d)|=|\sigma\Oseen\cap \S_I(X^d)|.
\end{alignat*}
Applying Theorem \ref{mainthm} with $\Lambda=\sigma\Oseen$ and using Lemma \ref{minest1} yields
\begin{alignat*}1
\left|N(\Oseen_I(K/K),X)-\frac{2^{s_K}\Vol \S_I(X^{d})}{(\sqrt{|\Delta_K|}[\Oseen_K:\Oseen])}\right|
\leq c(1,d) \frac{(\Log X)^{\q}X^{d-1}}{\eO^{d-1}}.
\end{alignat*}
This proves Theorem \ref{mainthm2} for $n=1$.\\

Now we assume $n>1$.
Combining Lemma \ref{lemmamainest}, (\ref{Naslatticepts}) and (\ref{primptsdecom}) yields for $\q=0$
\begin{alignat*}1
\left|N(\Oseen_I^n(K/k),X)-\frac{2^{s_Kn}\Vol \S_I(X^d)}{(\sqrt{|\Delta_K|}[\Oseen_K:\Oseen])^n}\right|
\ll \sum_{g\in G(K/k)}\frac{X^{dn-1}}{\eO^{dn-1}\delta_g(K/k)^{\mug}}.
\end{alignat*}
For $\q>0$ we additionally use (\ref{ZIpart}) to get
\begin{alignat*}1
\left|N(\Oseen_I^n(K/k),X)-\frac{2^{s_Kn}\Vol \S_I(X^d)}{(\sqrt{|\Delta_K|}[\Oseen_K:\Oseen])^n}\right|
\ll \sum_{\m_{\F}} \sum_{g\in G(K/k)}\frac{X^{dn-1}}{\eO^{dn-1}\delta_g(K/k)^{\mug}}.
\end{alignat*}
By Lemma \ref{m-est} we know $|\m_{\F}|\ll (\log^+X)^{\q}$, and this completes the proof of Theorem \ref{mainthm2}.

\section{Proof of Corollary \ref{cor2}}\label{Corollary2}
Recall that $\Oseen_{\emptyset}^n(K/k)=\emptyset$, and thus $N(\Oseen^n(K/k),X)=\sum_I N(\Oseen_I^n(K/k),X)$, where the sum runs over all
subsets of $\{1,\ldots,r_K+s_K\}$. Also recall the definition of $\Sempty(X^d)$ from (\ref{defSI}). 
As the $2^{r+s}$ sets $\S_I(T)$ define a partition of $\ST$ we see that Corollary \ref{cor2} follows immediately from 
Theorem \ref{mainthm2} and the following lemma.
\begin{lemma}
Suppose $X\geq 1$ and either $n>1$ or $K=k$. Then
\begin{alignat*}1
\left|N(\Oseen_{\emptyset}^n(K/k),X)-\frac{2^{s_Kn}\Vol \Sempty(X^d)}{(\sqrt{|\Delta_K|}[\Oseen_K:\Oseen])^n}\right|
\ll \sum_{g\in G(K/k)}\frac{1}{\eO^{dn-1}\delta_g(K/k)^{\mug}}.
\end{alignat*}
\end{lemma}
\begin{proof}
We have $\Oseen_{\emptyset}^n(K/k)=\emptyset$, $\Vol \Sempty(X^d)=(2^{r_K}\pi^{s_K})^n$ and 
$\det \sigma \Oseen=2^{-s_K}\sqrt{|\Delta_K|}[\Oseen_K:\Oseen]$.
As $\eO\delta_g(K/k)\geq 1$ and $\mug=mn(e-g)-1$ it suffices to show that for some $g\in G(K/k)$
\begin{alignat}1\label{forq0est}
\sqrt{|\Delta_K|}[\Oseen_K:\Oseen]\gg\eO^d\delta_g(K/k)^{m(e-g)}.
\end{alignat}
Let $\phi$ be the identity on $\IR^r\times\IC^{s}$, let 
$\lambda_i$ be as in (\ref{visuccmin}), and let $l$ be as in Definition \ref{ldef}. Then 
\begin{alignat*}1
\sqrt{|\Delta_K|}[\Oseen_K:\Oseen]\gg\lambda_1\cdots \lambda_d\geq \lambda_1^{l-1}\lambda_l^{d-l+1}. 
\end{alignat*}
If $l=1$ then $K=k$ and $\delta_g(K/k)=1$, so that (\ref{forq0est}) follows from the above and Lemma \ref{minest1}.
If $l\geq 2$ we take $g=[k(\theta_1/\theta_1,\ldots,\theta_{l-1}/\theta_1):k]\in G(K/k)$.
Applying Lemma \ref{minest1}, Lemma \ref{lemma2.7} and Lemma \ref{lemma2.6} yields (\ref{forq0est}), and thereby proves the lemma.
\end{proof}

\section{Volume computations}\label{SectionVolComp}
\begin{lemma}\label{VolCompSI}
Suppose $\q\geq 0$ and $T\geq 1$. Then we have
\begin{alignat*}1
\Vol \S_I(T)=2^{rn}\pi^{sn}(-1)^{\q}\left(-1+T^n\sum_{i=0}^{\q}\frac{(-\log T^n)^i}{i!}\right).
\end{alignat*}
\end{lemma}
\begin{proof}
Put $\r=|I\cap\{1,\ldots,r\}|$ and $\s=|I\cap\{r+1,\ldots,r+s\}|$.
From (\ref{SITcartprod}) we see that $\Vol \S_I(T)$ is given by the product of $2^{(r-\r)n}\pi^{(s-\s)n}$ and the $d'n$-dimensional volume of the set  
$\{(\z_i)_I\in \prod_I{K_i};\prod_I |\z_i|_{\infty}^{d_i}\leq T,|\z_i|_{\infty}\geq 1 \text{ for } i\in I\}$. Denote the latter by $V_{\r,\s}(T)$. 
For the sake of readability let us momentarily rewrite the variables $\z_i$ for $i\in I\cap\{1,\ldots,r\}$ as 
$\x_1,\ldots,\x_{\r}$ and  $\z_i$ for $i\in I\cap\{r+1,\ldots,r+s\}$ as 
$\y_1,\ldots,\y_{\s}$.
Clearly, we have
$V_{0,1}(T)=\pi^n(T^n-1)$, and Fubini's Theorem implies 
\begin{alignat*}1
V_{0,\s}(T)&=\int_{1\leq |\y_{\s}|_{\infty}\leq \sqrt{T}}V_{0,\s-1}(T/|\y_{\s}|_{\infty}^2)d\y_{\s}\\
&=n\int_{1\leq |y_{\s1}|\leq \sqrt{T}}\int_{0\leq |y_{\s2}|\leq |y_{\s1}|}\cdots \int_{0\leq |y_{\s n}|\leq |y_{\s1}|}V_{0,\s-1}(T/|y_{\s1}|^2)dy_{\s n}\cdots dy_{\s1}\\
&=n\int_{1\leq |y_{\s1}|\leq \sqrt{T}}(\pi |y_{\s1}|^2)^{n-1}V_{0,\s-1}(T/|y_{\s1}|^2)dy_{\s1}\\
&=n\int_{1}^{\sqrt{T}}\int_{0}^{2\pi}\rho (\pi \rho^2)^{n-1} V_{0,\s-1}(T/\rho^2)d\theta d\rho\\
&=2\pi^n n\int_{1}^{\sqrt{T}}\rho^{2n-1} V_{0,\s-1}(T/\rho^2)d\rho.
\end{alignat*}
By induction we conclude 
\begin{alignat*}1
V_{0,\s}(T)&=\pi^{\s n}\left((-1)^{\s}+\sum_{i=0}^{\s-1}\frac{(-1)^{\s-1-i}n^i}{i!}{T^n}(\log T)^i\right).
\end{alignat*}
Again, by Fubini's Theorem we find
\begin{alignat*}1
V_{\r,\s}(T)&=\int_{1\leq |\x_{\r}|_{\infty}\leq T}V_{\r-1,\s}(T/|\x_{\r}|_{\infty})d\x_{\r}\\
&=2^n n\int_{1}^{T} x_{\r1}^{n-1}V_{\r-1,\s}(T/x_{\r1})dx_{\r1}.
\end{alignat*}
Once more a simple induction argument shows
\begin{alignat*}3
V_{\r,\s}(T)&=2^{\r n}\pi^{\s n}\left((-1)^{\q-1}+\sum_{i=0}^{\q}\frac{(-1)^{\q-i}n^i}{i!}T^n(\log T)^i\right)\\
&=2^{\r n}\pi^{\s n}(-1)^{\q}\left(-1+T^n\sum_{i=0}^{\q}\frac{(-\log T^n)^i}{i!}\right).
\end{alignat*}
As $\Vol \S_I(T)=2^{(r-\r)n}\pi^{(s-\s)n}V_{\r,\s}(T)$ the lemma is proved.
\end{proof}

\begin{lemma}\label{VolCompS}
Suppose $T\geq 1$. Then we have
\begin{alignat*}1
\Vol \ST=\sum_{i=0}^{q} c_i T^n(\log T^n)^i,
\end{alignat*}
where
\begin{alignat*}1
c_i=\frac{2^{rn}\pi^{sn}}{i!}{q\choose i}.
\end{alignat*}
\end{lemma}
\begin{proof}
Clearly, we have $\Vol \ST=\sum_{I}\Vol \S_I(T)$, where the sum runs over all subsets $I$ of $\{1,\ldots,r+s\}$.
Now in order to compute the coefficient $c_i$ we have to sum the contribution
from each $\Vol \S_I(T)$. First note that 
\begin{alignat*}1
2^{rn}\pi^{sn}\sum_{I}(-1)^{\q+1}=2^{rn}\pi^{sn}\sum_{j=0}^{q+1}(-1)^{j}{q+1\choose j}=0.
\end{alignat*}
It remains to compute the coefficients $c_i$. The contribution of $\Vol \S_I(T)$ is zero if $\q=|I|-1<i$, and
\begin{alignat*}1
2^{rn}\pi^{sn}\frac{(-1)^{\q+i}}{i!}
\end{alignat*}
if $\q\geq i$. As we have ${q+1\choose \q+1}$ sets $I$ of cardinality $\q+1$ we conclude
\begin{alignat*}1
c_i=\frac{2^{rn}\pi^{sn}}{i!}\sum_{\q=i}^{q}(-1)^{i+\q}{q+1\choose \q+1}=\frac{2^{rn}\pi^{sn}}{i!}{q\choose i}.
\end{alignat*}
This concludes the proof of the lemma.
\end{proof}

\section{Upper bounds for the non-projectively primitive points}\label{nonproprimpts}
Recall the definition of the set of non-projectively primitive points in $\Oseen_K^n$
\begin{alignat*}1
\Oseen_{npp}^n(K/k)=\{\valpha\in \Oseen_K^n\backslash \Oseen_K^n(K/k); k(\valpha)=K\}.
\end{alignat*}
Let $k(n;e)$ be the subset of $\kbar^n$ of points $\valpha$ with $[k(\valpha):k]=e$.
Schmidt \cite[Theorem]{22} has shown the following estimate:
\begin{alignat}1\label{Schmidtbound}
N(k(n;e),X)\leq c_2(m,e,n)X^{me(n+e)},
\end{alignat}
where $c_2(m,e,n)=2^{me(e+n+3)+e^2+n^2+10e+10n}$.
\begin{lemma}\label{lemmanonproprimpts}
Suppose $e>1$. Then we have 
\begin{alignat*}1
\sum_{\coll_e(k)}N(\Oseen_{npp}^n(K/k),X)\ll\sup_{g\mid e}X^{m(g^2+gn+e^2/g+e)},
\end{alignat*}
where the supremum runs over all positive divisors $g<e$ of $e$. Moreover, for $e=1$ (and $X\geq 1$) we have
\begin{alignat*}1
\sum_{\coll_e(k)}N(\Oseen_{npp}^n(K/k),X)=1.
\end{alignat*}
\end{lemma}
\begin{proof}
If $e=1$ then $\coll_e(k)=\{k\}$ and $\Oseen_{npp}^n(k/k)=\{\0\}$. As $X\geq 1$ the lemma holds.
From now on we assume $e>1$.
Then the left-hand side counts points $\valpha=(\alpha_1,\ldots,\alpha_n)$ in $\Oseen_k(e;n)$ with
$k(\ldots,\alpha_i/\alpha_j,\ldots)\subsetneq k(\valpha)$ and $H(\valpha)\leq X$.
First suppose $n=1$. Then the left-hand side simply counts algebraic integers of 
degree $e$ over $k$ and height no larger than $X$. The number of these is by (\ref{Schmidtbound})
\begin{alignat*}1
\leq c_2(m,e,1)X^{me(e+1)}\ll \sup_{g\mid e}X^{m(g^2+gn+e^2/g+e)}.
\end{alignat*}
This proves the lemma for $n=1$.
Now we assume $n>1$.
As $e>1$ each $\valpha$ is nonzero, and so we loose only a factor $n$ if we assume 
$\alpha_1\neq 0$. Under this assumption $\valpha$ has the form $\valpha=(\theta,\theta \beta_2,\ldots,\theta \beta_n)$ 
such that with $F=k(\ldots,\alpha_i/\alpha_j,\ldots)$ one has: $k(\valpha)=F(\theta)$ and $k(\beta_2,\ldots,\beta_n)=F$. 
Furthermore, we have
\begin{alignat*}1
X\geq H(\valpha)=H(\theta,\theta \beta_2,\ldots,\theta \beta_n)=H(1/\theta,\beta_2,\ldots,\beta_n)\geq \max\{H(\theta),H(\beta_2,\ldots,\beta_n)\}.
\end{alignat*}
Therefore, it suffices to give an upper bound for the number of
$(\beta_2,\ldots,\beta_n,\theta)\in \kbar^n$ with
\begin{alignat*}1
&[k(\beta_2,\ldots,\beta_n):k]=g\leq e/2,\\
&[k(\theta,\beta_2,\ldots,\beta_n):k(\beta_2,\ldots,\beta_n)]=e/g,\\
&H(\beta_2,\ldots,\beta_n), H(\theta)\leq X.
\end{alignat*}
Let us fix a $g$ as above.
From (\ref{Schmidtbound}) we obtain the upper bound
\begin{alignat}1\label{boundbetavector}
c_2(m,g,n-1)X^{mg(g+n)}
\end{alignat}
for the number of such vectors $(\beta_2,\ldots,\beta_n)$.
Next for each $(\beta_2,\ldots,\beta_n)$ we count the number of $\theta$.
Now we have 
$[k(\theta,\beta_2,\ldots,\beta_n):k(\beta_2,\ldots,\beta_n)]=e/g$,
and, moreover, $H(\theta)\leq X$. Applying (\ref{Schmidtbound}) 
once more yields the upper bound
\begin{alignat}1
\label{ubalpha2}
c_2(mg,e/g,1) X^{[k(\beta_2,\ldots,\beta_n):\IQ](e/g)(e/g+1)}\ll X^{me(e/g+1)}
\end{alignat}
for the number of $\theta$, provided $(\beta_2,\ldots,\beta_n)$ is fixed.
Multiplying the bound (\ref{boundbetavector}) for the number
of $(\beta_2,\ldots,\beta_n)$ and (\ref{ubalpha2}) for the number
of $\theta$ gives the upper bound
\begin{alignat*}1
\ll X^{m(g^2+gn+e^2/g+e)}
\end{alignat*}
for the number of tuples $(\beta_2,\ldots,\beta_n,\theta)$. Taking the supremum over all possible values of $g$ 
proves the lemma.
\end{proof}

\section{Proof of Theorem \ref{mainthm1}}\label{proofmainthm1}
We start with a simple lemma. Put
\begin{alignat}1\label{gamma_g}
\gamma_g=m(g^2+g+e^2/g+e).
\end{alignat}
We remind the reader that $\mug=mn(e-g)-1$ and $\Cem=\max\{2+\frac{4}{e-1}+\frac{1}{m(e-1)},7-\frac{e}{2}+\frac{2}{me}\}$.
\begin{lemma}\label{lemmanecond}
Suppose $e>1$, $n>e+\Cem$ and $1\leq g\leq e/2$. Then we have
\begin{alignat}1
\label{firstcond}\gamma_g-\mug&\leq -2/e,\\
\label{secondcond}m(g^2+gn+e^2/g+e)&\leq men-1,\\
\label{thirdcond}(e+2)/4-n/2&\leq -\Cem/2.
\end{alignat}
\end{lemma}
\begin{proof}
Let us write (\ref{firstcond}) as
\begin{alignat*}1
m(g^2+g+e^2/g+e)-mn(e-g)+1+2/e\leq 0.
\end{alignat*}
With
\begin{alignat*}1
F(g)=\frac{g^2+g+e^2/g+e}{e-g}+\frac{1}{m(e-g)},
\end{alignat*}
this means
\begin{alignat}1\label{ngeqFplus}
n\geq F(g)+\frac{2}{me(e-g)}.
\end{alignat}
As $F(g)$ is a fraction with denominator dividing $mg(e-g)$ we conclude that
$n>F(g)$ implies $n\geq F(g)+\frac{1}{mg(e-g)}\geq F(g)+\frac{2}{me(e-g)}$.
Hence, it suffices to check $n>F(g)$.
Using that $(e-g)e^2/g^3\geq e^2/g^2$ for $1\leq g\leq e/2$, one sees
that the second derivative $F''(g)$ is positive for $1\leq g\leq e/2$.
Hence, $F(g)$ is here concave, and so it suffices to check that 
$n>F(1)$ and $n>F(e/2)$, which is equivalent to our hypothesis $n>e+\Cem.$
The claim (\ref{secondcond}) is equivalent to
\begin{alignat}1\label{ngeqFminus}
n\geq F(g)-\frac{g}{e-g}.
\end{alignat}
But we have just seen that (\ref{ngeqFplus}) holds and thus 
(\ref{ngeqFminus}) holds as well.
And, finally, (\ref{thirdcond}) follows from the assumptions 
$n>e+\Cem$ and $e>1$. This proves the lemma.
\end{proof}

We have the following disjoint union
\begin{alignat*}1
\Oseen_k(n;e)=\bigcup_{\coll_e(k)}\Oseen_K^n(K/k)\cup\Oseen_{npp}^n(K/k).
\end{alignat*}
Therefore,
\begin{alignat}1\label{sumdecomp}
N(\Oseen_k(n;e),X)=\sum_{\coll_e(k)}N(\Oseen_K^n(K/k),X)+\sum_{\coll_e(k)}N(\Oseen_{npp}^n(K/k),X).
\end{alignat}
Combining Lemma \ref{lemmanonproprimpts} and Lemma \ref{lemmanecond} shows that for $e>1$ and $n>e+\Cem$
\begin{alignat*}1
\sum_{\coll_e(k)}N(\Oseen_{npp}^n(K/k),X)\ll X^{men-1}.
\end{alignat*}
But by Lemma \ref{lemmanonproprimpts} the latter remains trivially true for $e=1$.
Therefore, we may focus on the first sum in (\ref{sumdecomp}).
By virtue of Corollary \ref{cor2} and Lemma \ref{VolCompS} it suffices to show that the following sums converge
\begin{alignat}1
\label{summainterms}&\sum_{\coll_e(k)}{|\Delta_K|^{-n/2}},\\
\label{sumerrorterms}&\sum_{\coll_e(k)} \sum_{g\in G(K/k)}\delta_g(K/k)^{-\mug}.
\end{alignat}
First suppose $e=1$. Then $\coll_e(k)=\{k\}$ consists of a single field, and, hence, both sums converge.
Next we assume 
\begin{alignat*}1
e>1,
\end{alignat*}
and thus by hypothesis $n>e+\Cem$.
Let us start with the sum in (\ref{summainterms}). Let
\begin{alignat*}1
N_{\Delta}(\coll_e(k),T)=|\{K\in \coll_e(k);|\Delta_K|\leq T\}|
\end{alignat*}
be the number of fields in $\coll_e(k)$ with discriminant no larger than $T$ in absolute value.
Schmidt \cite{55} has shown that
\begin{alignat}1\label{discbound}
N_{\Delta}(\coll_e(k),T)\leq c(k,e)T^{(e+2)/4}.
\end{alignat}
Ellenberg and Venkatesh \cite{56} have established a better bound for large values of $e$. However, for our purpose
Schmidt's bound is good enough. A simple dyadic summation argument proves the desired convergence. More precisely,
\begin{alignat*}3
\nonumber\sum_{\coll_e(k)}{|\Delta_K|^{-n/2}}
&=\sum_{i=1}^{\infty}\sum_{K\in \coll_e(k)\atop 2^{i-1}\leq |\Delta_K|<2^i}{|\Delta_K|^{-n/2}}
&&\leq \sum_{i=1}^{\infty}\frac{N_{\Delta}(\coll_e(k),2^i)}{2^{(i-1)n/2}}\\
&\leq c(k,e)\sum_{i=1}^{\infty}\frac{2^{i(e+2)/4}}{2^{(i-1)n/2}}&&=c(k,e)2^{n/2}\sum_{i=1}^{\infty}{2^{i((e+2)/4-n/2)}}.
\end{alignat*}
By (\ref{thirdcond}) we have $(e+2)/4-n/2\leq -\Cem/2<0$. Therefore, the last sum converges, and this proves the convergence of
(\ref{summainterms}).\\

To deal with the sum (\ref{sumerrorterms}) we need some more notation and an analogue of (\ref{discbound}) for the counting function
associated to $\delta_g$.
We define
\begin{alignat*}1
G_u=\bigcup_{\coll_e(k)}G(K/k).
\end{alignat*}
Clearly, $G_u\subset \{1,\ldots,[e/2]\}$. Now for any $g\in G_u$ we define 
\begin{alignat*}1
\coll_e^{(g)}(k)=\{K\in \coll_e(k); g\in G(K/k)\}
\end{alignat*}
and its counting function
\begin{alignat*}1
N_{\delta_g}(\coll_e^{(g)}(k),T)=|\{K\in \coll_e^{(g)}(k);\delta_g(K/k)\leq T\}|.
\end{alignat*}
\begin{lemma}
For $g$ in $G_u$, and $\gamma_g$ as in (\ref{gamma_g}) we have
\begin{alignat*}1
N_{\delta_g}(\coll_e^{(g)}(k),T)\ll T^{\gamma_g}.
\end{alignat*}
\end{lemma}
\begin{proof}
Since $H(\alpha_1,\alpha_2)\geq \max\{H(\alpha_1),H(\alpha_2)\}$ it suffices to show that the number of tuples $(\alpha_1,\alpha_2)\in \kbar^2$ with
\begin{alignat*}1
&[k(\alpha_1):k]=g,\\
&[k(\alpha_1,\alpha_2):k(\alpha_1)]=e/g,\\
&H(\alpha_1), H(\alpha_2)\leq T
\end{alignat*}
is $\ll T^{\gamma_g}$.
But the latter can be shown exactly in the same manner as in the proof of Lemma \ref{lemmanonproprimpts}.
\end{proof}
Now we can show the convergence of (\ref{sumerrorterms}). We proceed similar as for (\ref{summainterms}).
\begin{alignat*}3
\sum_{K\in \coll_e(k)}\sum_{g\in G(K/k)}\delta_g(K/k)^{-\mug}&=
\sum_{g\in G_u}\sum_{K\in \coll_e^{(g)}(k)}\delta_g(K/k)^{-\mug}\\
&=\sum_{g\in G_u}\sum_{i=1}^{\infty}\sum_{K\in \coll_e^{(g)}(k)\atop 2^{i-1}\leq \delta_g(K/k)<2^i}\delta_g(K/k)^{-\mug}\\
&\leq \sum_{g\in G_u}\sum_{i=1}^{\infty}\frac{N_{\delta_g}(\coll_e^{(g)}(k),2^i)}{2^{(i-1)\mug}}\\
&\ll \sum_{g\in G_u}\sum_{i=1}^{\infty}2^{i(\gamma_g-\mug)}.
\end{alignat*}
By (\ref{firstcond}) we have $\gamma_g-\mug\leq -2/e$, and this proves the convergence of (\ref{sumerrorterms}).
Therefore, the proof of Theorem \ref{mainthm1} is complete.

\section*{Acknowledgements}
I would like to thank the referees for carefully reading the manuscript and for providing various helpful comments.

\bibliographystyle{amsplain}
\bibliography{literature}

\end{document}